\documentclass[twocolumn]{autart}

\usepackage{graphicx}
\usepackage{epsfig} 
\usepackage{times} 
\usepackage{amsmath}
\usepackage{amssymb}  
\usepackage{color}
\usepackage{amsfonts}
\usepackage{subfigure}
\usepackage{multirow}
\usepackage{multicol}
\usepackage{siunitx}
\usepackage{overpic}
\usepackage{mathrsfs}  
\usepackage{wrapfig}

\usepackage{booktabs}
\usepackage{threeparttable}

\usepackage{epstopdf}
\newtheorem{example}{Example}
\newtheorem{lemma}{Lemma}

\newtheorem{theorem}{Theorem}
\newtheorem{definition}{Definition}

\usepackage{xspace}

\newenvironment{proof}[1][Proof]{\begin{trivlist}
\item[\hskip \labelsep {\bfseries #1}]}{\end{trivlist}}

\begin{document}

%\title{\LARGE \bf Nonlinear Bilateral Output-Feedback Control of a Viscous Hamilton-Jacobi PDE}

\begin{frontmatter}

\title{Nonlinear Bilateral Output-Feedback Control \\ for a Class of Viscous Hamilton-Jacobi PDEs}

%\title{\LARGE \bf Compensation of Quasilinear First-Order Hyperbolic PDE Actuator Dynamics for Nonlinear Systems}

%\title{\LARGE \bf Compensation of Actuator Dynamics Governed by \\ Scalar Conservation Laws}

%\title{Compensation of Actuator Dynamics Governed by \\ Quasilinear Hyperbolic PDEs}

\author[nikau]{Nikolaos Bekiaris-Liberis\thanksref{footnoteinfo}}\ead{nikos.bekiaris@dssl.tuc.gr} and  
\author[rafa]{Rafael Vazquez}\ead{rvazquez1@us.es}
%\thanks{N. Bekiaris-Liberis is with the Department of Production Engineering \& Management, Technical University of Crete, Chania, 73100, Greece. Email address: \texttt{nikos.bekiaris@dssl.tuc.gr}. }

\address[nikau]{Department of Production Engineering \& Management, Technical University of Crete, Chania, 73100, Greece.}

\address[rafa]{Department of Aerospace Engineering, University of Seville, Seville 41092, Spain.}            

\thanks[footnoteinfo]{Corresponding author.}

%\thanks{I. Karafyllis is with the Department of Mathematics, National Technical University of Athens, Zografou Campus, 15780, Athens, Greece. Email address: \texttt{iasonkar@central.ntua.gr}.}

%\thanks{M. Krstic is with the Department of Mechanical \& Aerospace Engineering, University of California, San Diego, La Jolla, CA, 92093, USA. Email address: \texttt{krstic@ucsd.edu}.} }

\maketitle
\thispagestyle{empty}     
\pagestyle{empty}

\begin{abstract}
%\baselineskip=1.7\normalbaselineskip
% mention firsts-order mention scalar conservation laws
We tackle the boundary control and estimation problems for a class of viscous {Hamilton-Jacobi} PDEs, considering bilateral actuation and sensing, i.e., at the two boundaries of a 1-D spatial domain. First, we solve the nonlinear trajectory generation problem for this type of PDEs, providing the necessary feedforward actions at both boundaries. Second, in order to guarantee trajectory tracking with an arbitrary decay rate, we construct nonlinear, full-state feedback laws employed at the two boundary ends. Third, a nonlinear observer is constructed, using measurements from both boundaries, which is combined with the full-state feedback designs into an observer-based output-feedback law. All of our designs are explicit since they are constructed interlacing a feedback linearizing transformation (which we introduce) with backstepping. Due to the fact that the linearizing transformation is locally invertible, only regional stability results are established, which are, nevertheless, accompanied with  region of attraction estimates. Our stability proofs are based on the utilization of the linearizing transformation together with the employment of backstepping transformations, suitably formulated to handle the case of bilateral actuation and sensing. We illustrate the developed methodologies via application to traffic flow control and we present consistent simulation results.
\end{abstract}

\end{frontmatter}

\interdisplaylinepenalty=2500 %*** I didn't use this

\section{Introduction}
\label{intro}
\subsection{Motivation}
Contrary to linear parabolic Partial Differential Equations (PDEs), for which explicit boundary control and estimation designs are now largely available, see, for instance, \cite{smyshlyaev}, \cite{meurer book}, in the nonlinear case, the design of explicit boundary control and estimation schemes is a more challenging problem. In addition, specific engineering applications, such as, for example, vehicular traffic \cite{pachroo}, \cite{treiber}, plasma systems \cite{frederick}, \cite{frederick1}, fluids \cite{balogh3}, \cite{byrnes}, chemical reactors \cite{meurer book}, heat exchangers \cite{meurer book}, and litium-ion batteries \cite{tang1}, \cite{tang2}, to name only a few, call for the development of systematic control and estimation design methodologies that, besides being able to efficiently exploit the capabilities of the available actuators and sensors, they can also be made fault tolerant. Motivated by scalar, conservation law models for vehicular traffic flow that include a viscous term, in order to account for drivers' look-ahead ability \cite{pachroo}, \cite{treiber}, we consider the problems of boundary control and estimation of a certain class of viscous Hamilton-Jacobi (HJ) PDEs, which constitutes an alternative macroscopic description of traffic flow dynamics \cite{bayen1}, \cite{newell}. In particular, we consider the case in which actuation and sensing is available at both boundaries (which we refer to as ``bilateral" in our control and estimation approaches), aiming at constructing control and estimation schemes capable of utilizing efficiently both the available actuators and the available measurements.

%with respect to actuators and sensors failures, as well as robust to noise and control saturation. 

%systematic?

%for more general class?

%explicit maybe also closed-form or analytical? 

\subsection{Literature}
Arguably, the most relevant results to the ones presented here are those dealing with the controller and observer designs for viscous Burgers-type PDEs, which may be viewed as conservation law counterparts of the class of viscous HJ PDEs with quadratic Hamiltonian considered here. The trajectory generation problem for certain forms of viscous Burgers equations is considered in \cite{krstic1}, \cite{meurer1}, \cite{save}, whereas full-state boundary feedback laws are designed in \cite{meurer fred}, \cite{iwamoto}, \cite{krstic alone}, \cite{krstic magnis1}, \cite{liu}, \cite{andrey meurer}. Observers and output-feedback controllers are presented in \cite{balogh1}, \cite{balogh2}, \cite{byrnes}, \cite{jada}, \cite{krstic1}. Explicit boundary control and observer designs for other nonlinear parabolic PDEs also exist, see, e.g., \cite{hasan general}, \cite{hasan general1}, \cite{meurer tac}, \cite{vasq1}, \cite{vasq2}. Although it is a different problem, for completeness, it should be mentioned that the control design problem of inviscid versions of Burgers or of specific HJ PDEs is considered in, e.g., \cite{aubin}, \cite{blandin}, \cite{bayen1}, \cite{krstic alone}. Bilateral controllers and observers for certain classes of linear parabolic and hyperbolic PDEs are recently developed in \cite{auriol}, \cite{auriol1}, \cite{Vazquez1}, \cite{vazquez2}. We should also mention here that, in comparison to \cite{bekiaris1}, in the present paper we consider, 1) a more general class of viscous HJ PDE systems, 2) the problems of trajectory generation and tracking, and 3) the problems of bilateral control and estimation.

%the present paper includes substantially new contributions in comparison to \cite{bekiaris1} in the present paper we, 1) consider a more general class of viscous HJ PDE systems, 2) solve the trajectory generation and tracking problems, 3) construct bilateral, nonlinear full-state feedback laws, 4) introduce bilateral, nonlinear observers and observer-based output feedback laws.  

%construct bilateral observes and controllers

\subsection{Results}
Our contributions are summarized as follows. First, we solve the nonlinear trajectory generation problem for the considered viscous HJ PDE, providing explicit feedforward actions at both boundaries.  The key ingredient in our approach is the employment of a feedback linearizing transformation (inspired by the Hopf-Cole transformation \cite{cole}, \cite{hopf}) that we introduce, which allows us to convert the original nonlinear problem to a motion planning problem for a linear heat equation. We then establish the well-posedness of the feedforward controllers for the original nonlinear PDE system, for reference outputs that belong to Gevrey class (of certain order) with sufficiently small magnitude.

%appropriately restricting 

% show that for reference outputs in Gevrey class (of certain order) the order of Gevrey regularity is preserved with a constant whose size may be controlled by controlling the size of the reference outputs. This allows us to establish the well-posedness of the feedforward controllers for the original nonlinear PDE system.

Second, we design full-state feedback laws in order to achieve trajectory tracking, with an arbitrary decay rate, as the system is not, in general, asymptotically stable around a given reference trajectory. Modifying, in a suitable way, the introduced feedback linearizing transformation we recast the original nonlinear control problem to a problem of full-state feedback stabilization of a linear heat equation, with Neumann actuation at each of the two boundaries. The bilateral boundary controllers are designed using the recently introduced backstepping technique \cite{Vazquez1}. We then establish local asymptotic stability of the closed-loop system in $H^1$ norm, employing a Lyapunov functional and we provide an estimate of the region of attraction of the controller. Our stability result is local in $H^1$ norm due to the fact that the linearizing transformation is invertible only locally and, in particularly, the size of the supremum norm of the transformed PDE state should be appropriately restricted.

Third, we turn our attention to the observer-based output-feedback trajectory tracking problem. We design a nonlinear observer, employing boundary measurements from both ends of the spatial domain. The observer design is based on the introduced linearizing transformation and on a suitable formulation of the backstepping methodology in \cite{vazquez2} to the case of a one-dimensional spatial domain. We then show that the bilateral, observer-based output-feedback controller achieves local asymptotic stabilization of the reference trajectory in $H^1$ norm.

Finally, we apply the developed methodologies to a model of highway traffic flow. We illustrate, in simulation, the effectiveness of the proposed control design technique, including also a comparison with the unilateral case (i.e., the case in which a full-state feedback controller is applied only at the one boundary). In general, less control effort is required in the bilateral case, fact that may be useful in actual implementations.

\subsection{Organization}
We start presenting the class of viscous HJ PDEs under consideration and introducing the feedback linearizing transformation in Section~\ref{problem formulation}. We then continue in a way such that a reader interested only in the designs could skip the details of the proofs. Specifically, in Section~\ref{sec gen} we present the nonlinear feedforward control designs. In Section~\ref{sec full-state} we present the nonlinear, full-state feedback controllers and in Section~\ref{sec full-state sta} we prove local asymptotic stability of the closed-loop system. In Section~\ref{sec observer} we present the nonlinear observer design and in Section~\ref{sec observer sta} we prove stability of the closed-loop system under the observer-based output-feedback laws. We present an example of traffic flow control in Section~\ref{traffic app}. Concluding remarks and directions of future research are provided in Section~\ref{conclude}.

\subsection{Notation and Definitions} 
We use the common definition of class $\mathcal{K}$, $\mathcal{K}_{\infty}$ and $\mathcal{KL}$ functions from \cite{khalil}. For a function $u\in L^2(0,1)$ we denote by $\|u(t)\|_{L^2}$ the norm $\|u(t)\|_{L^2}=\sqrt{\int_0^1u(x,t)^2dx}$. For $u\in H^1(0,1)$ we denote by $\|u(t)\|_{H^1}$ the norm $\|u(t)\|_{H^1}=\sqrt{\int_0^1u(x,t)^2dx}+\sqrt{\int_0^1u_x(x,t)^2dx}$. We denote by $C^j(A)$ the space of functions that have continuous derivatives of order $j$ on $A$. We denote an initial condition as $u_0(x)=u(x,t_0)$ with some $t_0\geq0$, for all $x\in[0,1]$. With $C\left([t_0,\infty);H^2\left(0,1\right)\right)$ we denote the class of continuous mappings on $[t_0,\infty)$ with values into $H^2\left(0,1\right)$. We denote by $C_T^{2,1}\left([0,1]\times (t_0,T)\right)$ the space of functions that have continuous spatial derivatives of order $2$ and continuous time derivatives of order $1$ on $[0,1]\times (t_0,T)$, and define $C_{\infty}^{2,1}\!=\!C^{2,1}$. %

\begin{definition}
The function $f(t)$ belongs to $G_{F,M,\gamma}\left(\mathbb{S}\right)$, the Gevrey class of order $\gamma$ in $\mathbb{S}$, if $f(t)\in C^{\infty}\left(\mathbb{S}\right)$ and there exist positive constants $F$, $M$ such that $\sup_{t\in\mathbb{S}}\left|{f}^{(n)}(t)\right|\leq FM^{n}\left(n!\right)^{\gamma}$, for all $n=0,1,2,\ldots$.
\end{definition}

\section{Problem Formulation and Feedback Linearization}
\label{problem formulation}
We consider the following viscous HJ PDE system 
\begin{eqnarray}
u_t(x,t)&=&\epsilon u_{xx}(x,t)-au_x(x,t)\left(b+u_x(x,t)\right)\label{sys1}\\
u_x(0,t)&=&U_0(t)\label{sys2}\\
u_x\left(1,t\right)&=&U_1(t),\label{sys3}
\end{eqnarray}
where $u$ is the PDE state, $x\in[0,1]$ is the spatial variable, $t\geq t_0\geq0$ is time, $\epsilon>0$ is a viscosity coefficient, $a\neq0$ and $b\in\mathbb{R}$ are constant parameters, and $U_0$, $U_1$ are control variables. We introduce next a feedback linearizing transformation, which allows us to convert the problems of trajectory generation and tracking for the nonlinear HJ PDE (\ref{sys1})--(\ref{sys3}) to the corresponding problems for a linear diffusion-advection PDE.

%\subsection{Feedback linearizing transformation}
The following locally invertible transformation
\begin{eqnarray}
\bar{v}(x,t)=e^{-\frac{a}{\epsilon}{u}(x,t)}-1,\label{transformation1}
\end{eqnarray}
and the control laws
\begin{eqnarray}
{U}_0(t)&=&-\frac{\epsilon}{a} e^{\frac{a}{\epsilon}{u}(0,t)}\bar{V}_0(t)\label{til1}\\
{U}_1(t)&=&-\frac{\epsilon}{a}e^{\frac{a}{\epsilon}{u}(1,t)}\bar{V}_1(t),\label{til2}
\end{eqnarray}
where $\bar{V}_0$, $\bar{V}_1$ are the new control variables yet to be chosen, transform system (\ref{sys1})--(\ref{sys3}) to
\begin{eqnarray}
\bar{v}_t(x,t)&=&\epsilon \bar{v}_{xx}(x,t)-ab\bar{v}_x(x,t)\label{sys0erv}\\
\bar{v}_x(0,t)&=&\bar{V}_0(t)\\
\bar{v}_x(1,t)&=&\bar{V}_1(t).\label{sysnerv}
\end{eqnarray}
It turns out that in the control design and analysis it is more convenient to perform an additional transformation, namely 
\begin{eqnarray}
v(x,t)=\bar{v}(x,t)e^{-\frac{ab}{2\epsilon}x}, \label{spat}
\end{eqnarray}
in order to re-write (\ref{sys0erv})--(\ref{sysnerv}) as
\begin{eqnarray}
{v}_t(x,t)&=&\epsilon {v}_{xx}(x,t)-\frac{a^2b^2}{4\epsilon}{v}(x,t)\label{sys0ervtra}\\
{v}_x(0,t)&=&{V}_0(t)\label{syssys}\\
{v}_x(1,t)&=&{V}_1(t),\label{sysnervtra}
\end{eqnarray}
where
\begin{eqnarray}
\bar{V}_0(t)&=&V_0(t)+\frac{ab}{2\epsilon}\bar{v}(0,t)\label{bar v0}\\
\bar{V}_1(t)&=&e^{\frac{ab}{2\epsilon}}V_1(t)+\frac{ab}{2\epsilon}\bar{v}(1,t),\label{bar v1}
\end{eqnarray}
and $V_0$, $V_1$ are the new control variables.

\section{Trajectory Generation}
\label{sec gen}
%\subsection{Feedforward Boundary Control Design}
%We start by considering the following backstepping transformation \cite{smyshlyaev}
%\begin{eqnarray}
%
%\end{eqnarray}

%Transforming the original system (\ref{sys1})--(\ref{sys3}) into (\ref{sys0ervtra})--(\ref{sysnervtra}) enables us to employ the following reference trajectory
In this section we design the feedforward boundary control laws that generate the desired reference outputs. We solve the problem first for the linearized system (\ref{sys0ervtra})--(\ref{sysnervtra}) and we then provide the feedforward actions for the original system (\ref{sys1})--(\ref{sys3}). We consider as outputs of the system the values $u\left(x_0,t\right)$ and $u_x\left(x_0,t\right)$, where $x_0$ is some fixed point within the interval $[0,1]$. %Yet, it would be trivial one to extend the results of the present section to the case where the outputs are $u(0,t)$ and $u_x(0,t)$. Note that since the values of $u_x(1,t)$ could be directly assigned by $U_1(t)$, one could view the boundary condition (\ref{sys3}) as an extra degree of freedom, which may be chosen such that a, potentially secondary, control objective is achieved. 

\begin{theorem}
\label{theorem 1}
Let $y_1^{\rm r}(t)$ and $y_2^{\rm r}(t)$ be in $G_{F,M,\gamma}\left([0,+\infty)\right)$ class with $1\leq\gamma<2$. There exists a positive constant $\mu_1$ such that if $F\leq\mu_1$ then the functions 
\begin{eqnarray}
u^{\rm r}(x,t)&=&-\frac{\epsilon}{a}\ln\left(e^{\frac{ab}{2\epsilon}x}v^{\rm r}(x,t)+1\right)\label{166}\\
{U}^{\rm r}_0(t)&=&-\frac{\epsilon}{a} \frac{v_x^{\rm r}(0,t)+\frac{ab}{2\epsilon}v^{\rm r}(0,t)}{1+v^{\rm r}(0,t)}\label{til1tra}\\
{U}^{\rm r}_1(t)&=&-\frac{\epsilon e^{\frac{ab}{2\epsilon}}}{a} \frac{v_x^{\rm r}(1,t)+\frac{ab}{2\epsilon}v^{\rm r}(1,t)}{1+e^{\frac{ab}{2\epsilon}}v^{\rm r}(1,t)},\label{til2tra}
%{U}^{\rm r}_1(t)&=&-\frac{\epsilon}{a}e^{\frac{a}{\epsilon}{y}_1^{\rm r}(t)}\left(e^{\frac{ab}{2\epsilon}}y_{2,v}^{\rm r}(t)+\frac{ab}{2\epsilon}\left(e^{-\frac{a}{\epsilon}y_1^{\rm r}(t)}-1\right)\right),\label{til2tra}
\end{eqnarray}
where
\begin{eqnarray}
%v^{\rm r}(x,t)&=&\sum_{k=0}^{\infty}\left(\mathcal{D}^{k}\left\{y_{1,v}^{\rm r}(t)\right\}\frac{\left(x-1\right)^{2k}}{2k!}+\mathcal{D}^{k}\left\{y_{2,v}^{\rm r}(t)\right\}\frac{\left(x-1\right)^{2k+1}}{\left(2k+1\right)!}\right)\label{ref12}\\
v^{\rm r}(x,t)&=&\sum_{k=0}^{\infty}\frac{1}{\epsilon^k} \frac{\left(x-x_0\right)^{2k}}{\left(2k\right)!}\sum_{m=0}^{k}\binom {k} {m}\left(\frac{a^2b^2}{4\epsilon}\right)^{k-m}\nonumber\\
&&\times{y_{1,v}^{\rm r}}^{(m)}(t)+\sum_{k=0}^{\infty}\frac{1}{\epsilon^k} \frac{\left(x-x_0\right)^{2k+1}}{\left(2k+1\right)!}\nonumber\\
&&\times\sum_{m=0}^{k}\binom {k} {m}\left(\frac{a^2b^2}{4\epsilon}\right)^{k-m}{y_{2,v}^{\rm r}}^{(m)}(t)\label{ref12}\\
y_{1,v}^{\rm r}(t)&=&e^{-\frac{ab}{2\epsilon}x_0}\left(e^{-\frac{a}{\epsilon}y_1^{\rm r}(t)}-1\right)\label{y1v}\\
y_{2,v}^{\rm r}(t)&=&e^{-\frac{ab}{2\epsilon}x_0}\left(\vphantom{\frac{ab}{2\epsilon}\left(e^{-\frac{a}{\epsilon}y_1^{\rm r}(t)}-1\right)}-\frac{a}{\epsilon}e^{-\frac{a}{\epsilon}y_1^{\rm r}(t)}y_2^{\rm r}(t)\right.\nonumber\\
&&\left.-\frac{ab}{2\epsilon}\left(e^{-\frac{a}{\epsilon}y_1^{\rm r}(t)}-1\right)\right),\label{y2v}
\end{eqnarray}
satisfy the boundary value problem (\ref{sys1})--(\ref{sys3}) and, in particular, $u^{\rm r}\left(x_0,t\right)=y_1^{\rm r}(t)$ and $u_x^{\rm r}\left(x_0,t\right)=y_2^{\rm r}(t)$. %Moreover, .
\end{theorem}

%and the operator ${D}^{k}\left\{y(t)\right\}$ is defined as ${D}^{0}\left\{y(t)\right\}=y(t)$ and ${D}^{k}\left\{y(t)\right\}={D}\left\{{D}^{k-1}\left\{y(t)\right\}\right\}$, $k=1,2,\ldots$, with ${D}\left\{y(t)\right\}=\frac{1}{\epsilon}\left(\dot{y}(t)+\frac{a^2b^2}{4\epsilon}y(t)\right)$,

\begin{proof}
Via transformations (\ref{transformation1}) and (\ref{spat}), in order to generate the desired trajectory $u^{\rm r}(x,t)$ and to provide the feedforward laws ${U}^{\rm r}_0(t)$, ${U}^{\rm r}_1(t)$, which achieve $u^{\rm r}\left(x_0,t\right)=y_1^{\rm r}(t)$ and $u_x^{\rm r}\left(x_0,t\right)=y_2^{\rm r}(t)$, it is sufficient to generate $v^{\rm r}(x,t)$ that satisfies (\ref{sys0ervtra}) with
\begin{eqnarray}
v^{\rm r}\left(x_0,t\right)&=&y_{1,v}^{\rm r}(t)\label{ref v1}\\
v^{\rm r}_x\left(x_0,t\right)&=&y_{2,v}^{\rm r}(t),\label{ref v2}
\end{eqnarray}
where $y_{1,v}^{\rm r}(t)$ and $y_{2,v}^{\rm r}(t)$ are defined in (\ref{y1v}) and (\ref{y2v}), respectively. The feedforward laws ${U}^{\rm r}_0(t)$, ${U}^{\rm r}_1(t)$ are then given combining (\ref{syssys}), (\ref{sysnervtra}) with (\ref{bar v0}), (\ref{bar v1}) and (\ref{til1}), (\ref{til2}). Moreover, $v^{\rm r}(x,t)$ should be restricted appropriately such that (\ref{166}) and (\ref{til1tra}) are well-posed, which holds true whenever 
\begin{eqnarray}
\sup_{x\in[0,1]}\left|v^{\rm r}(x,t)\right|<\bar{c}e^{-\left|\frac{ab}{2\epsilon}\right|},\quad \mbox{for all $t\geq t_0$},\label{Res1}
\end{eqnarray}
for some constant $\bar{c}\in(0,1)$, in addition to $v_x^{\rm r}(x,t)$ being bounded for all $x\in[0,1]$ and $t\geq t_0$.

Since system (\ref{sys0ervtra})--(\ref{sysnervtra}) is in the form of a linear diffusion-advection PDE we postulate the reference trajectory $v^{\rm r}$ in the form, see, e.g., \cite{laroche}, \cite{meurer1}, \cite{meurer2}
\begin{eqnarray}
v^{\rm r}(x,t)=\sum_{k=0}^{\infty}\alpha_k(t)\frac{\left(x_0-x\right)^k}{k!},\label{ref1}
\end{eqnarray}
where the functions $\alpha_k(t)$, $k=0,1,\ldots$ are yet to be determined in order for (\ref{ref1}) to satisfy (\ref{sys0ervtra}) as well as (\ref{ref v1}) and (\ref{ref v2}). Substituting (\ref{ref1}) into (\ref{sys0ervtra}) we arrive at the following recursive relation for $\alpha$'s
\begin{eqnarray}
\alpha_{k+2}(t)&=&\frac{1}{\epsilon}\left(\dot{\alpha}_k(t)+\frac{a^2b^2}{4\epsilon}\alpha_k(t)\right)\\
\alpha_0(t)&=&y_{1,v}^{\rm r}(t)\\
\alpha_1(t)&=&-y_{2,v}^{\rm r}(t),
\end{eqnarray}
and thus, (\ref{ref1}) may be written as in (\ref{ref12}), see, e.g., \cite{laroche}, \cite{meurer1}, \cite{meurer2}.  Employing the results in, for example, \cite{meurer1} (Remark 4), one can conclude that the series (\ref{ref12}) is convergent (with an infinite radius of convergence) provided that $y_{1,v}^{\rm r}(t)$ and $y_{2,v}^{\rm r}(t)$ belong to $G_{F_1^*,M_1^*,\gamma}\left([0,+\infty)\right)$ for $1\leq\gamma<2$, for some positive constants $F_1^*$ and $M_1^*$. 

We derive next explicit Gevrey-type estimates for $y_{1,v}^{\rm r}(t)$ and $y_{2,v}^{\rm r}(t)$ as, in order to guarantee that condition (\ref{Res1}) holds, one has to guarantee that, in addition to the functions $y_{1,v}^{\rm r}(t)$ and $y_{2,v}^{\rm r}(t)$ belonging to $G_{{F}_1^*,{M}_1^*,\gamma}\left([0,+\infty)\right)$ for $1\leq\gamma<2$, that the constant ${F}_1^*$ may be chosen sufficiently small when $F$ is sufficiently small. Toward that end, from Lemmas~\ref{lemma1} and \ref{lemma2} in Appendix A we obtain that
\begin{eqnarray}
\sup_{t\geq 0}\left|{y_{1,v}^{\rm r}}^{(n)}(t)\right|&\leq&\bar{F}_1 \bar{M}_1^{n}\left(n!\right)^{\gamma}, \quad \mbox{for all $n=0,1,\ldots$}\label{nik1}\\
\sup_{t\geq 0}\left|{y_{2,v}^{\rm r}}^{(n)}(t)\right|&\leq&\bar{F}_2\bar{M}_2^{n}\left(n!\right)^{\gamma}, \quad \mbox{for all $n=0,1,\ldots$},\label{nik2}
\end{eqnarray}
where 
\begin{eqnarray}
\bar{F}_1&=&F\frac{|a|}{\epsilon}e^{F\frac{|a|}{\epsilon}}e^{-\frac{ab}{2\epsilon}x_0}\label{31}\\
\bar{M}_1&=&Me^{F\frac{|a|}{\epsilon}}\label{contshit1}\\
\bar{F}_2&=&F\frac{|a|}{\epsilon}e^{F\frac{|a|}{\epsilon}}e^{-\frac{ab}{2\epsilon}x_0}\left(e^{-F\frac{|a|}{\epsilon}}+ \frac{|ab|}{2\epsilon} + F\frac{|a|}{\epsilon}\right)\label{32} \\
\bar{M}_2&=&\left(1+F\frac{|a|}{\epsilon}e^{F\frac{|a|}{\epsilon}}\right)\bar{M}_1,\label{contshit2}%e^{\bar{M}_1}\max\left\{\bar{M}_1,2\right\}\\
%\bar{R}_3&=&R\frac{|a|}{\epsilon}\left(1+R\frac{|a|}{\epsilon}e^{R\frac{|a|}{\epsilon}}\right)e^{-\frac{ab}{2\epsilon}x_0}\\
%\bar{M}_3&=&\left(1+R\frac{|a|}{\epsilon}e^{R\frac{|a|}{\epsilon}}\right)\bar{M}_1.
%\bar{M}_3&=&\bar{M}_2\max\left\{\frac{|ab|}{2\epsilon}e^{-\frac{ab}{2\epsilon}x_0},1\right\}\\
%\bar{M}_4&=&\bar{M}_1e^{\bar{M}_1}\left(1+\bar{M}_1e^{\bar{M}_1}\right)\max\left\{e^{-\frac{ab}{2\epsilon}x_0},1\right\}\label{32}.
%\bar{M}_5&=&\bar{M}_1e^{\bar{M}_1}\left(1+\bar{M}_1e^{\bar{M}_1}\right)\label{32}.
%\bar{M}_4&=&\bar{M}_1e^{\bar{M}_1}\max\left\{\bar{M}_1,2\right\}.\label{32}
\end{eqnarray}
and hence, one can choose $F_1^*=F\frac{|a|}{\epsilon}e^{F\frac{|a|}{\epsilon}}e^{-\frac{ab}{2\epsilon}x_0}$ $\times\max\left\{1,e^{-F\frac{|a|}{\epsilon}}+ \frac{|ab|}{2\epsilon} + F\frac{|a|}{\epsilon}\right\}$ and $M_1^*=\bar{M}_2$. Consequently, series (\ref{ref12}) is convergent. Moreover, combining (\ref{ref12}) and (\ref{nik1}), (\ref{nik2}) we get that
\begin{eqnarray}
\left|v^{\rm r}(x,t)\right|&\leq&  \bar{F}_1\sum_{k=0}^{\infty} \frac{1}{\epsilon^k}\left(k!\right)^{\gamma-2}\left(\frac{a^2b^2}{4\epsilon}+\bar{M}_1\right)^k\nonumber\\
&&+\bar{F}_2\sum_{k=0}^{\infty}\frac{1}{\epsilon^k}\left(k!\right)^{\gamma-2}\left(\frac{a^2b^2}{4\epsilon}+\bar{M}_2\right)^k,\label{ref12new}
\end{eqnarray}
where we used the fact that $\left(k!\right)^2\leq \left(2k\right)!$. For all $x\in[0,1]$ the general term, say $\zeta_k$, in the first series satisfies 
\begin{eqnarray}
\left|\frac{\zeta_{k+1}}{\zeta_k}\right|= \frac{1}{\epsilon}\left(\frac{a^2b^2}{4\epsilon}+\bar{M}_1\right)\left(k+1\right)^{\gamma-2},\end{eqnarray}
and thus, since $\gamma<2$, we conclude that $\lim_{k\to\infty}\left|\frac{\zeta_{k+1}}{\zeta_k}\right|=0<1$, which in turn implies, employing D'Alembert's criterion, that the infinite sum converges to a positive number, say $l_1$. Similarly, the second infinite sum converges to a positive number, say $l_2$. Therefore, from (\ref{ref12new}) we arrive at
\begin{eqnarray}
\left|v^{\rm r}(x,t)\right|&\leq&  \max\left\{\bar{F}_1,\bar{F}_2\right\}\left(l_1+l_2\right),\nonumber\\
&&\mbox{for all $x\in[0,1]$ and $t\geq t_0$}\label{ref12new1},
\end{eqnarray}
and hence, by choosing $\mu_1$ such that $\max\left\{\bar{F}_1,\bar{F}_2\right\}\left(l_1+l_2\right)$ $<\bar{c}e^{-\left|\frac{ab}{2\epsilon}\right|}$, for some constant $\bar{c}\in(0,1)$, which, according to relations (\ref{31}), (\ref{32}) is always possible (note that $l_1$, $l_2$ are continuous functions of $F$ since the two series in (\ref{ref12new}) converge uniformly and from (\ref{contshit1}), (\ref{contshit2}) it follows that $\bar{M}_1$, $\bar{M}_2$ are continuous with respect to $F$), condition (\ref{Res1}) is satisfied. It follows from (\ref{166}) that $u^{\rm r}(x,t)$ is uniformly bounded with respect to time and spatial variable. The uniform boundedness, with respect to time and spatial variable, of $v_x^{\rm r}(x,t)$, $v_{xx}^{\rm r}(x,t)$, and $v_t^{\rm r}(x,t)$, which, from (\ref{166}) and (\ref{Res1}), imply the uniform boundedness of $u_x^{\rm r}(x,t)$, $u_{xx}^{\rm r}(x,t)$, and $u_t^{\rm r}(x,t)$, follow by differentiating (\ref{ref12}) and employing almost identical arguments (see also, e.g., Section 3 in \cite{laroche}). \qed

\end{proof}

\begin{example}
\rm{Consider system (\ref{sys1})--(\ref{sys3}) with $a=-1$, $b=0$ and assume that the desired reference trajectories are $y_1^{\rm r}(t)=0$ and $y_2^{\rm r}(t)=d\textrm{sin}(t)$, where $d>0$. For sufficiently small $d$ the conditions of Theorem \ref{theorem 1} are satisfied. The reference trajectory as well as the reference inputs are given by
\begin{eqnarray}
u^{\rm r}(x,t)&=&\epsilon\ln\left(1+g_1(x,t)\right)\label{166rr}\\
g_1(x,t)&=&\frac{d}{2\sqrt{\epsilon}}e^{\frac{x-x_0}{\sqrt{2\epsilon}}}\textrm{sin}\left(t+\frac{x-x_0}{\sqrt{2\epsilon}}-\frac{\pi}{4}\right)\nonumber\\
&&-\frac{d}{2\sqrt{\epsilon}}e^{\frac{x_0-x}{\sqrt{2\epsilon}}}\textrm{sin}\left(t+\frac{x_0-x}{\sqrt{2\epsilon}}-\frac{\pi}{4}\right)\\
{U}^{\rm r}_0(t)&=&\frac{d}{2}\frac{e^{-\frac{x_0}{\sqrt{2\epsilon}}}\textrm{sin}\left(t-\frac{x_0}{\sqrt{2\epsilon}}\right)}{1+g_1(0,t)}\nonumber\\
&&+\frac{d}{2} \frac{e^{\frac{x_0}{\sqrt{2\epsilon}}}\textrm{sin}\left(t+\frac{x_0}{\sqrt{2\epsilon}}\right)}{1+g_1(0,t)}\label{til1trarr}\\
{U}^{\rm r}_1(t)&=&\frac{d}{2} \frac{e^{\frac{1-x_0}{\sqrt{2\epsilon}}}\textrm{sin}\left(t+\frac{1-x_0}{\sqrt{2\epsilon}}\right)}{1+g_1(1,t)}\nonumber\\
&&+\frac{d}{2}\frac{e^{\frac{x_0-1}{\sqrt{2\epsilon}}}\textrm{sin}\left(t+\frac{x_0-1}{\sqrt{2\epsilon}}\right)}{1+g_1(1,t)},\label{til2trarr}
%{U}^{\rm r}_1(t)&=&-\frac{\epsilon}{a}e^{\frac{a}{\epsilon}{y}_1^{\rm r}(t)}\left(e^{\frac{ab}{2\epsilon}}y_{2,v}^{\rm r}(t)+\frac{ab}{2\epsilon}\left(e^{-\frac{a}{\epsilon}y_1^{\rm r}(t)}-1\right)\right),\label{til2tra}
\end{eqnarray}
where we also used the fact that $\textrm{sin}\left(y\right)-\textrm{cos}\left(y\right)=\sqrt{2}\textrm{sin}\left(y-\frac{\pi}{4}\right)$, for any $y\in\mathbb{R}$. In Fig. \ref{fig tr 1} we show the generated trajectory $u^{\rm r}$ as well as its spatial derivative $u_x^{\rm r}$. 
\begin{figure}[t]
\centering
\includegraphics[width=\linewidth]{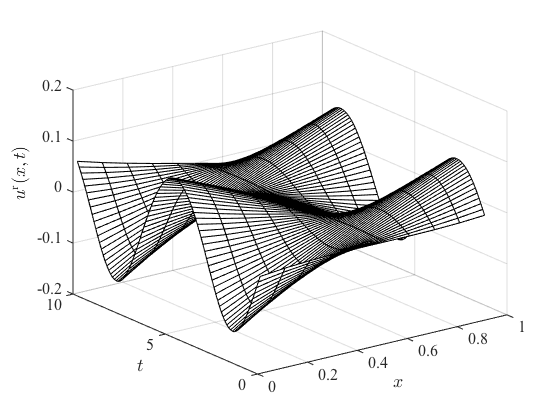}
\includegraphics[width=\linewidth]{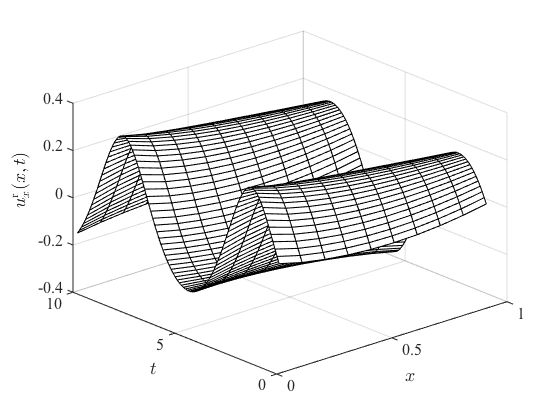}
\caption{Functions (\ref{166rr})--(\ref{til2trarr}) that solve the nonlinear trajectory generation problem for system (\ref{sys1})--(\ref{sys3}) with $a=-1$, $b=0$, and $\epsilon=0.5$, with reference trajectories $y_1^{\rm r}(t)=0$ and $y_2^{\rm r}(t)=0.25\textrm{sin}(t)$, for $x_0=\frac{1}{2}$.}
\label{fig tr 1}
\end{figure}

To see that the functions defined in (\ref{166rr})--(\ref{til2trarr}) solve the nonlinear trajectory generation problem note first that relation (\ref{ref12}) in the present case becomes
\begin{eqnarray}
%v^{\rm r}(x,t)&=&\sum_{k=0}^{\infty}\left(\mathcal{D}^{k}\left\{y_{1,v}^{\rm r}(t)\right\}\frac{\left(x-1\right)^{2k}}{2k!}+\mathcal{D}^{k}\left\{y_{2,v}^{\rm r}(t)\right\}\frac{\left(x-1\right)^{2k+1}}{\left(2k+1\right)!}\right)\label{ref12}\\
v^{\rm r}(x,t)&=&\frac{d}{\epsilon}\sum_{k=0}^{\infty}\frac{1}{\epsilon^k} \frac{\left(x-x_0\right)^{2k+1}}{\left(2k+1\right)!}\textrm{sin}^{(k)}(t)\label{ref12re1}.
\end{eqnarray}}Using the facts that $\textrm{sin}(t)=\frac{1}{2j}\left(e^{j t}-e^{-j t}\right)$, $j=\frac{1}{2}\left(1+j\right)^2$, and $-j=\frac{1}{2}\left(1-j\right)^2$ we obtain
\begin{eqnarray}
v^{\rm r}(x,t)&=&\frac{d}{2j\epsilon}\sum_{k=0}^{\infty}\frac{1}{\epsilon^k} \frac{\left(x-x_0\right)^{2k+1}}{\left(2k+1\right)!}\left(\frac{1}{2^k}\left(1+j\right)^{2k}e^{jt}\right.\nonumber\\
&&\left.-\frac{1}{2^k}\left(1-j\right)^{2k}e^{-jt}\right)\label{ref12re212},
\end{eqnarray}
and hence, employing the power series expansion for the hyperbolic sine we arrive at
\begin{eqnarray}
v^{\rm r}(x,t)&=&\frac{d}{\sqrt{2\epsilon}j}\left(\frac{1}{1+j}e^{j t}\textrm{sinh}\left(\frac{\left(x-x_0\right)\left(1+j\right)}{\sqrt{2\epsilon}}\right)\right.\nonumber\\
&&\left.-\frac{1}{1-j}e^{-j t}\textrm{sinh}\left(\frac{\left(x-x_0\right)\left(1-j\right)}{\sqrt{2\epsilon}}\right)\right).
\end{eqnarray}
Performing some tedious algebraic manipulations we get
\begin{eqnarray}
v^{\rm r}(x,t)&=&\frac{d}{2\sqrt{\epsilon}}e^{\frac{x-x_0}{\sqrt{2\epsilon}}}\textrm{sin}\left(t+\frac{x-x_0}{\sqrt{2\epsilon}}-\frac{\pi}{4}\right)\nonumber\\
&&-\frac{d}{2\sqrt{\epsilon}}e^{\frac{x_0-x}{\sqrt{2\epsilon}}}\textrm{sin}\left(t+\frac{x_0-x}{\sqrt{2\epsilon}}-\frac{\pi}{4}\right).\label{vb}
\end{eqnarray}
Equations (\ref{166rr})--(\ref{til2trarr}) are then derived combining (\ref{vb}) with (\ref{166})--(\ref{til2tra}).
\end{example}

%\subsection{Two reference outputs}

\section{Bilateral Full-State Feedback Boundary Control Design}
\label{sec full-state}
Having available the reference trajectory for system (\ref{sys1})--(\ref{sys3}), in this section, we design the boundary feedback laws that stabilize the desired reference trajectory for any initial condition. We start deriving the dynamics of the error between the actual and the reference states. We then introduce a feedback linearizing transformation for the tracking error's dynamics, which, in turn, enables us to design full-state feedback, boundary control laws utilizing infinite-dimensional backstepping for linear systems.

%\subsection{Feedforward Boundary Control Design}
%\subsection{Feedback Boundary Control Design}
\subsection{Tracking error dynamics and motivation for control}
We define the error variables 
\begin{eqnarray}
\tilde{u}(x,t)&=&u(x,t)-u^{\rm r}(x,t)\label{error u}\\ 
\tilde{U}_0(t)&=&U_0(t)-U_0^{\rm r}(t)\\
\tilde{U}_1(t)&=&U_1(t)-U_1^{\rm r}(t).
\end{eqnarray}
Differentiating (\ref{error u}) with respect to $t$ and $x$, using the fact that $u^{\rm r}(x,t)$ satisfies system (\ref{sys1})--(\ref{sys3}) we get that $\tilde{u}$ satisfies the following system
\begin{eqnarray}
\tilde{u}_t(x,t)&=&\epsilon \tilde{u}_{xx}(x,t)-a\tilde{u}_x(x,t)\left(b+\tilde{u}_x(x,t)\right)\nonumber\\
&&-2au_x^{\rm r}(x,t)\tilde{u}_x(x,t)\label{sys1er}\\
\tilde{u}_x(0,t)&=&\tilde{U}_0(t)\label{sys2er}\\
\tilde{u}_x\left(1,t\right)&=&\tilde{U}_1(t).\label{sys3er}
\end{eqnarray}
A feedback control design is needed to asymptotically stabilize the origin of (\ref{sys1er})--(\ref{sys3er}). To see this note that the zero solution of (\ref{sys1er})--(\ref{sys3er}) is not asymptotically stable since any constant could be an equilibrium of (\ref{sys1er})--(\ref{sys3er}).

\subsection{Feedback linearizing transformation for the tracking error dynamics}
Guided from the feedback linearizing transformation (\ref{transformation1}) we define
\begin{eqnarray}
\tilde{\bar{v}}(x,t)=e^{-\frac{a}{\epsilon}\tilde{u}(x,t)}-1,\label{transformation1er}
\end{eqnarray}
which it is readily shown that satisfies the following PDE
\begin{eqnarray}
\tilde{\bar{v}}_t(x,t)&=&\epsilon \tilde{\bar{v}}_{xx}(x,t)-a\left(b+2u_x^{\rm r}(x,t)\right)\tilde{\bar{v}}_x(x,t)\label{sys0erver}\\
\tilde{\bar{v}}_x(0,t)&=&\tilde{\bar{V}}_0(t)\label{convtil1}\\
\tilde{\bar{v}}_x(1,t)&=&\tilde{\bar{V}}_1(t),\label{sysnerver}
\end{eqnarray}
where we choose
\begin{eqnarray}
\tilde{U}_0(t)&=&-\frac{\epsilon}{a} e^{\frac{a}{\epsilon}\tilde{u}(0,t)}\tilde{\bar{V}}_0(t)\label{til1er}\\
\tilde{U}_1(t)&=&-\frac{\epsilon}{a}e^{\frac{a}{\epsilon}\tilde{u}(1,t)}\tilde{\bar{V}}_1(t),\label{til2er}
\end{eqnarray}
and $\tilde{\bar{V}}_0(t)$, $\tilde{\bar{V}}_1(t)$ are new control variables. With the additional transformation (see also Fig. \ref{inter})
\begin{eqnarray}
\tilde{{v}}(x,t)=\tilde{\bar{v}}(x,t)e^{-\frac{ab}{2\epsilon}x-\frac{a}{\epsilon}u^{\rm r}(x,t)},\label{transformation1add}
\end{eqnarray}
and selecting the control variables $\tilde{\bar{V}}_0(t)$, $\tilde{\bar{V}}_1(t)$ as
\begin{eqnarray}
\tilde{\bar{V}}_0(t)&=&  e^{\frac{a}{\epsilon}{u^{\rm r}}(0,t)} \tilde{V}_0(t)\nonumber\\
&&+\frac{a}{\epsilon}\left(\frac{b}{2}+u_x^{\rm r}(0,t)\right)\left(e^{-\frac{a}{\epsilon}\tilde{u}(0,t)}-1\right) \label{til1er1}\\
\tilde{\bar{V}}_1(t)&=&  e^{\frac{ab}{2\epsilon}+\frac{a}{\epsilon}{u^{\rm r}}(1,t)} \tilde{V}_1(t)\nonumber\\
&&+\frac{a}{\epsilon}\left(\frac{b}{2}+u_x^{\rm r}(1,t)\right)\left(e^{-\frac{a}{\epsilon}\tilde{u}(1,t)}-1\right) \label{til2er1},
\end{eqnarray}
we arrive at the following system
\begin{eqnarray}
\tilde{v}_t(x,t)&=&\epsilon \tilde{v}_{xx}(x,t)-\frac{a^2b^2}{4\epsilon}\tilde{v}(x,t)\label{error1}\\
\tilde{v}_x(0,t)&=&\tilde{V}_0(t)\label{shh}\\
\tilde{v}_x(1,t)&=&\tilde{V}_1(t),\label{error2}
\end{eqnarray}
where the control variables $\tilde{V}_0(t)$ and $\tilde{V}_1(t)$ are chosen later on (in Section \ref{subback}) via the backstepping methodology. 

\begin{figure}
\centering
\vspace{0.2in}
\begin{overpic}[width=\linewidth]{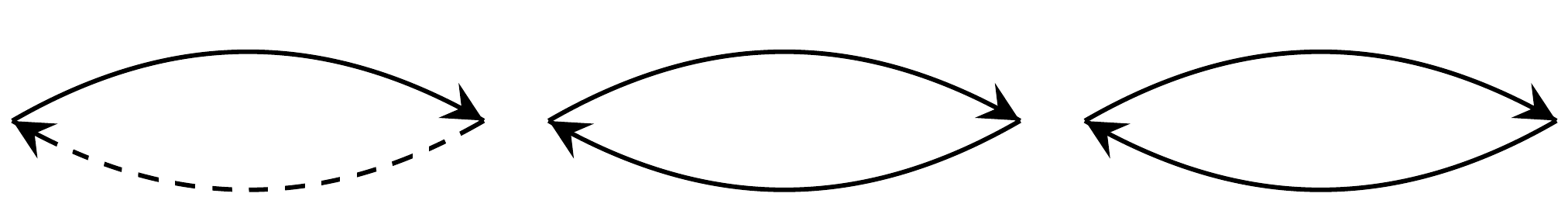}
%\put(30,12){$V$}
%\put(32,15){Input from humans}
%\put(29,60){Direction of convection}
\put(8,13){$e^{-\frac{a}{\epsilon}(\cdot)}-1$}
\put(38,13){$e^{-\frac{ab}{2\epsilon}x-\frac{a}{\epsilon}u^{\rm r}(x,t)}$}
\put(40,-3){$e^{\frac{ab}{2\epsilon}x+\frac{a}{\epsilon}u^{\rm r}(x,t)}$}
\put(4,-3){$-\frac{\epsilon}{a}\log\left(\cdot+1\right)$}
\put(79,13){$\mathscr{K}\{\cdot\}$}
\put(80,-3){$\mathscr{L}\{\cdot\}$}
\put(-2,6){$\tilde{u}$}
\put(32,6){$\tilde{\bar{v}}$}
\put(66.5,6){$\tilde{v}$}
\put(100,6){${w}$}
\end{overpic}
\vspace{0.1in}
\caption{The interconnections between $\tilde{u}$, $\tilde{\bar{v}}$, $\tilde{v}$, and $w$ involved in transformations (\ref{transformation1er}), (\ref{transformation1add}), and (\ref{trans1}). The operators $\mathscr{K}\{\cdot\}$ and $\mathscr{L}\{\cdot\}$ are defined as $\mathscr{K}\{\tilde{v}\}(x)=\tilde{v}(x)-\int_{-x+1}^xk(x,y)\tilde{v}(y)dy$ and $\mathscr{L}\{w\}(x)=w(x)+\int_{-x+1}^xl(x,y)w(y)dy$ respectively. The leftmost transformation (i.e., transformation (\ref{transformation1er})) is only locally invertible.}
\label{inter}
\end{figure}
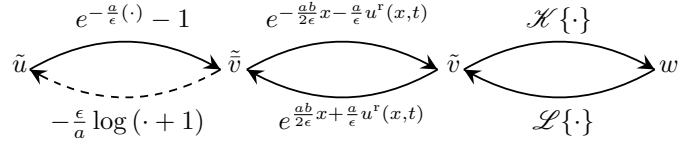

Note that system (\ref{error1})--(\ref{error2}), besides being linear, does not incorporate any spatially- or time-dependent terms, which may be the case when considering trajectory tracking problems for nonlinear systems. This is possible here because the overall feedback linearizing transformation (\ref{transformation1add}) may be expressed as the difference of two nonlinear functions of $u$ and $u^{\rm r}$, which both satisfy the linear PDE (\ref{sys0ervtra}) (or, equivalently, (\ref{error1})) since both $u$ and $u^{\rm r}$ satisfy (\ref{sys1}). Moreover, relations (\ref{shh}), (\ref{error2}) are derived differentiating (\ref{transformation1add}) with respect to $x$ and using (\ref{convtil1}), (\ref{sysnerver}) as well as defining the new control inputs $\tilde{V}_0$, $\tilde{V}_1$ according to (\ref{til1er1}), (\ref{til2er1}).% For the reader's benefit, we provide details on the computations that lead to system (\ref{error1})--(\ref{error2}) in Appendix  B.

%where the control variables $\tilde{V}_0(t)=V_0(t)-V_0^{\rm r}(t)$, $\tilde{V}_1(t)=V_1(t)-V_0^{\rm r}(t)$, and $V_0^{\rm r}(t)$, $V_1^{\rm r}(t)$ are given in relations (\ref{}) and (\ref{}), respectively. 

%Therefore, one may employ the feedback linearizing transformation (\ref{transformation1}) as well as transformation (\ref{spat}) where $u$ is replaced by $\tilde{u}$, $\bar{v}$ is replaced by $\tilde{\bar{v}}=\bar{v}-\bar{v}^{\rm r}$, and $v$ is replaced by $\tilde{v}=v-v^{\rm r}$. It can then be shown that $\tilde{v}$ satisfies
%\begin{eqnarray}
%\tilde{v}_t(x,t)&=&\epsilon \tilde{v}_{xx}(x,t)-\frac{a^2b^2}{4\epsilon}\tilde{v}(x,t)\label{error1}\\
%\tilde{v}_x(0,t)&=&\tilde{V}_0(t)\\
%\tilde{v}_x(1,t)&=&\tilde{V}_1(t),\label{error2}
%\end{eqnarray}
%where $\tilde{V}_0(t)=V_0(t)-V_0^{\rm r}(t)$, $\tilde{V}_1(t)=V_1(t)-V_0^{\rm r}(t)$, and $V_0^{\rm r}(t)$, $V_1^{\rm r}(t)$ are given in relations (\ref{}) and (\ref{}), respectively. 

\subsection{Bilateral boundary control design}
\label{subback}

%I%n order to achieve an arbitrary decay rate for the closed-loop system, in this subsection we design full-state feedback controllers. 

Exploiting the fact that the $\tilde{v}$ variable satisfies the linear diffusion-advection PDE (\ref{error1})--(\ref{error2}) we design the boundary feedback laws as \cite{Vazquez1} 
\begin{eqnarray}
\tilde{V}_0(t)&=&k\left(0,0\right)\tilde{v}\left(0,t\right)-\int_{0}^1k_x\left(0,\xi\right)\tilde{v}\left(\xi,t\right)d\xi\label{bac1}\\
\tilde{V}_1(t)&=&k\left(1,1\right)\tilde{v}\left(1,t\right)+\int_{0}^1k_x\left(1,\xi\right)\tilde{v}\left(\xi,t\right)d\xi,\label{bac2}
%\tilde{V}_0(t)&=&k\left(0,0\right)\tilde{v}\left(0,t\right)-\int_{0}^1k_x\left(0,\xi\right)\tilde{v}\left(\xi,t\right)d\xi\label{bac1}\\
%\tilde{V}_1(t)&=&k\left(1,1\right)\tilde{v}\left(1,t\right)+\int_{0}^1k_x\left(1,\xi\right)\tilde{v}\left(\xi,t\right)d\xi,\label{bac2}
\end{eqnarray}
where the kernel $k\left(x,\xi\right)$ is given explicitly, for $(x,\xi)$ in the domain $D=D_1\cup D_2$, where $D_1=\left\{\left(x,\xi\right):\frac{1}{2}\leq x\leq 1,\right.$ $\left.-x+1\leq\xi\leq x\right\}$ and $D_2=\left\{\left(x,\xi\right):0\leq x\leq \frac{1}{2},\right.$ $\left.x\leq\xi\leq 1-x\right\}$, by
\begin{eqnarray}
%k(x,\xi)&=&-\frac{1}{2}\sqrt{\frac{c_1}{\epsilon}}{\rm sgn} \left(x-\frac{1}{2}\right){\rm I}_{1}\left(\sqrt{\frac{c_1}{\epsilon}\left(\left(x-\frac{1}{2}\right)^2-\left(\xi-\frac{1}{2}\right)^2\right)}\right)\sqrt{\frac{x+\xi-1}{x-\xi}}\\
%r(x,\xi)&=&-\frac{1}{2}\sqrt{\frac{c_1}{\epsilon}}{\rm I}_{1}\left(\sqrt{\frac{c_1}{\epsilon}\left(\left(x-\frac{1}{2}\right)^2-\left(\xi-\frac{1}{2}\right)^2\right)}\right)\sqrt{\frac{x+\xi-1}{x-\xi}},\label{poo}.
k(x,\xi)&=&-\frac{1}{2}\sqrt{\frac{c_1}{\epsilon}}\frac{{\rm I}_{1}\left(\sqrt{\frac{c_1}{\epsilon}\left(\left(x-\frac{1}{2}\right)^2-\left(\xi-\frac{1}{2}\right)^2\right)}\right)}{\sqrt{\left(x-\frac{1}{2}\right)^2-\left(\xi-\frac{1}{2}\right)^2}}\nonumber\\
&&\times\left(x+\xi-1\right)\label{poo},
%\nonumber\\
%&=&k(x,\xi),K\left(x-\frac{1}{2},\xi-\frac{1}{2}\right)
\end{eqnarray}
with ${\rm I}_1$ denoting the modified Bessel function of the first kind of first order. Combining (\ref{til1er}), (\ref{til2er}) and (\ref{til1er1}), (\ref{til2er1}) with (\ref{bac1}), (\ref{bac2}) the boundary feedback laws in the original variables are written via (\ref{transformation1er}), (\ref{transformation1add}) as
\begin{eqnarray}
{U}_0(t)&=&-\frac{\epsilon}{a} e^{\frac{a}{\epsilon}{\tilde{u}}(0,t)}\left(\left(k\left(0,0\right)+\frac{ab}{2\epsilon}\right)\left(e^{-\frac{a}{\epsilon}\tilde{u}(0,t)}-1\right)\right.\nonumber\\
&&\left.\vphantom{\int_{0}^1K_x\left(-\frac{1}{2},\xi-\frac{1}{2}\right)}-e^{\frac{a}{\epsilon}{u^{\rm r}}(0,t)}\int_{0}^1k_x\left(0,\xi\right)e^{-\frac{ab}{2\epsilon}\xi-\frac{a}{\epsilon}u^{\rm r}\left(\xi,t\right)}\right.\nonumber\\
&&\left.\times\left(e^{-\frac{a}{\epsilon}\tilde{u}\left(\xi,t\right)}-1\right)d\xi\vphantom{\left(k\left(0,0\right)+\frac{ab}{2\epsilon}\right)\left(e^{-\frac{a}{\epsilon}\tilde{u}(0,t)}-1\right)}\right)+U_0^{\rm r}(t)e^{\frac{a}{\epsilon}{\tilde{u}}(0,t)}\label{con1}\\
{U}_1(t)&=&-\frac{\epsilon}{a} e^{\frac{a}{\epsilon}{\tilde{u}}(1,t)}\left(\left(k\left(1,1\right)+\frac{ab}{2\epsilon}\right)\left(e^{-\frac{a}{\epsilon}\tilde{u}(1,t)}-1\right)\right.\nonumber\\
&&\left.\vphantom{\int_{0}^1K_x\left(\frac{1}{2},\xi-\frac{1}{2}\right)}+e^{\frac{ab}{2\epsilon}+\frac{a}{\epsilon}u^{\rm r}\left(1,t\right)}\int_{0}^1k_x\left(1,\xi\right)e^{-\frac{ab}{2\epsilon}\xi-\frac{a}{\epsilon}u^{\rm r}\left(\xi,t\right)}\right.\nonumber\\
&&\left.\times\left(e^{-\frac{a}{\epsilon}\tilde{u}\left(\xi,t\right)}-1\right)d\xi\vphantom{\left(k\left(0,0\right)+\frac{ab}{2\epsilon}\right)\left(e^{-\frac{a}{\epsilon}\tilde{u}(0,t)}-1\right)}\right)+U_1^{\rm r}(t)e^{\frac{a}{\epsilon}{\tilde{u}}(1,t)},\label{con2}
\end{eqnarray}
where $U_0^{\rm r}(t)$ and $U_1^{\rm r}(t)$ are defined in (\ref{til1tra}) and (\ref{til2tra}), respectively, with the error variable $\tilde{u}$ being defined in (\ref{error u}) and the reference trajectory $u^{\rm r}$ being defined in (\ref{166}).

\section{Trajectory Tracking Under Full-State Feedback}
\label{sec full-state sta}
In order to show asymptotic stability of the closed-loop system, under the full-state feedback laws, in the original variable $\tilde{u}$ we have to ensure that the linearizing transformation (\ref{transformation1er}) is invertible. The inverse of transformation (\ref{transformation1er}) is given by (see also Fig. \ref{inter})
\begin{eqnarray}
\tilde{u}(x,t)=-\frac{\epsilon}{a}{\rm ln}\left(\tilde{\bar{v}}(x,t)+1\right),\label{transformationinv}
\end{eqnarray}
which is well-defined when the initial conditions and solutions of the system satisfy for some $c\in(0,1]$ 
\begin{eqnarray}
\sup_{x\in[0,1]}|\tilde{\bar{v}}(x,t)|<c,\quad \mbox{for all $t\geq t_0$}.\label{feasibility}
\end{eqnarray}
Due to the feasibility condition (\ref{feasibility}), only a local stability result can be obtained, which is stated next. 

\begin{theorem}
\label{theorem 2}
Consider a closed-loop system consisting of the plant (\ref{sys1})--(\ref{sys3}) and the control laws (\ref{con1}), (\ref{con2}). Under the conditions of Theorem \ref{theorem 1} for the reference outputs, there exist a positive constant $\mu$ and a class $\mathcal{KL}$ function $\beta$ such that for all initial conditions ${u}_0\in H^2(0,1)$ which are compatible with the feedback laws (\ref{con1}), (\ref{con2}) and which satisfy
%\begin{eqnarray}
%Let $c_1\geq0$ be arbitrary. 
%(which depends on $c_1$)
%\|\tilde{u}(0)\|_{H^1}&<& \alpha_1^{-1}\left(\frac{c}{2\mu_1}\right)\label{region}\\
%\alpha_1(s)&=&\frac{3s}{\epsilon}e^{\frac{2s}{\epsilon}},\label{defa1}
%\end{eqnarray}
%for some $0<c<1$, 
\begin{eqnarray}
\|\tilde{u}\left(t_0\right)\|_{H^1}&<&\mu\label{region},
\end{eqnarray}
the following holds
\begin{eqnarray}
\|\tilde{u}(t)\|_{H^1}&\leq& \beta\left(\|\tilde{u}\left(t_0\right)\|_{H^1},t-t_0\right),\quad\mbox{for all $t\geq t_0$}\label{thm1est}.
\end{eqnarray}
%\begin{eqnarray}
%\|\tilde{u}(t)\|_{H^1}&\leq& \alpha\left(\|\tilde{u}(0)\|_{H^1}\right)e^{-\left(c_1+\frac{1}{4\epsilon}\right) t},\quad\mbox{$\forall$ $t\geq0$}\label{thm1est}\\
%\alpha\left(s\right)&=&\frac{3\mu_1}{1-c}e^{\frac{2s}{\epsilon}}s.\label{thm1esjvjht}
%\end{eqnarray}
%Moreover, the closed-loop system has a unique solution $\tilde{u}\in H^{2,1}\left((0,1)\times (0,\infty)\right)$.
Moreover, the closed-loop system has a unique solution $u\in C\left([t_0,\infty);H^2(0,1)\right)$ with $u\in C^{2,1}\left([0,1]\times \left(t_0,\infty\right)\right)$.
%${u}\in C^{2,1}\left([0,1]\times (t_0,\infty)\right)$.
\end{theorem}

The proof of Theorem \ref{theorem 2} is based on the following three lemmas whose proofs can be found in Appendix B. Note that the compatibility conditions in the statement of Theorem \ref{theorem 2} are the following
\begin{eqnarray}
{u}_{0}'(0)&=&-\frac{\epsilon}{a} e^{\frac{a}{\epsilon}\left({u}_0(0)-u^{\rm r}\left(0,t_0\right)\right)}\left(\left(k\left(0,0\right)+\frac{ab}{2\epsilon}\right)\right.\nonumber\\
&&\left.\times\left(e^{-\frac{a}{\epsilon}\left({u}_0(0)-u^{\rm r}\left(0,t_0\right)\right)}-1\right)-e^{\frac{a}{\epsilon}{u^{\rm r}}\left(0,t_0\right)}\right.\nonumber\\
&&\!\left.\times\int_{0}^1k_x\left(0,\xi\right)e^{-\frac{ab}{2\epsilon}\xi}\left(e^{-\frac{a}{\epsilon}{u}_0\left(\xi\right)}-e^{-\frac{a}{\epsilon}u^{\rm r}\left(\xi,t_0\right)}\right)\right.\nonumber\\
&&\left.\times d\xi\vphantom{\left(k\left(0,0\right)+\frac{ab}{2\epsilon}\right)\left(e^{-\frac{a}{\epsilon}\tilde{u}(0,t)}-1\right)}\right)+u_x^{\rm r}\left(0,t_0\right)e^{\frac{a}{\epsilon}\left({u}_0(0)-u^{\rm r}\left(0,t_0\right)\right)}\\
{u}_{0}'(1)&=&-\frac{\epsilon}{a} e^{\frac{a}{\epsilon}\left({u}_0\left(1\right)-u^{\rm r}\left(1,t_0\right)\right)}\left(\left(k\left(1,1\right)+\frac{ab}{2\epsilon}\right)\right.\nonumber\\
&&\left.\times\left(e^{-\frac{a}{\epsilon}\left({u}_0\left(1\right)-u^{\rm r}\left(1,t_0\right)\right)}-1\right)+e^{\frac{ab}{2\epsilon}+\frac{a}{\epsilon}u^{\rm r}\left(1,t_0\right)}\right.\nonumber\\
&&\left.\times\int_{0}^1k_x\left(1,\xi\right)e^{-\frac{ab}{2\epsilon}\xi-\frac{a}{\epsilon}u^{\rm r}\left(\xi,t_0\right)}\right.\nonumber\\
&&\left.\times\left(e^{-\frac{a}{\epsilon}\left({u}_0\left(\xi\right)-u^{\rm r}\left(\xi,t_0\right)\right)}-1\right)d\xi\vphantom{\left(k\left(0,0\right)+\frac{ab}{2\epsilon}\right)\left(e^{-\frac{a}{\epsilon}\tilde{u}(0,t)}-1\right)}\right)\nonumber\\
&&+u_x^{\rm r}\left(1,t_0\right)e^{\frac{a}{\epsilon}\left({u}_0\left(1\right)-u^{\rm r}\left(1,t_0\right)\right)}.
%\tilde{u}_{0x}(1)\!=\!\epsilon\left(-r_2+k(1,1)\right)\left(e^{\frac{1}{\epsilon}\tilde{u}_0(1)}-1\right)+\frac{\epsilon e^{\frac{1}{\epsilon}\tilde{u}_0(1)}}{1+\sigma e^{-\frac{1}{2\epsilon}}} \int_0^1k_x(1,y)\left(e^{\frac{y-1}{2\epsilon}}\!+\!\sigma e^{-\frac{y}{2\epsilon}}\right)$ $\times\left(1-e^{-\frac{1}{\epsilon}\tilde{u}_0(y)}\right)dy.
\end{eqnarray}
%In fact, the compatibility conditions requirement may be relaxed

\begin{lemma}
\label{lemma 1thm2}
There exists a class $\mathcal{K}_{\infty}$ function $\alpha_1$ such that if $\tilde{u}\in H^1(0,1)$ then $\tilde{\bar{v}}\in H^1(0,1)$ and the following holds
\begin{eqnarray}
\|\tilde{\bar{v}}(t)\|_{H^1}\leq \alpha_1\left(\|\tilde{u}(t)\|_{H^1}\right).\label{lemrel1}
\end{eqnarray}
\end{lemma}

\begin{lemma}
\label{lemma 2thm2}
For all solutions of the system that satisfy (\ref{feasibility}) for some $0<c<1$, if $\tilde{\bar{v}}\in H^1(0,1)$ then $\tilde{u}\in H^1(0,1)$ and the following holds
\begin{eqnarray}
%\|\tilde{u}(t)\|_{H^1}\leq \alpha_3\left(\|\tilde{v}(t)\|_{H^1}\right).\label{lemrel1n}
\|\tilde{u}(t)\|_{H^1}\leq \frac{\epsilon}{|a|\left(1-c\right)}\|\tilde{\bar{v}}(t)\|_{H^1}.\label{lemrel1n}
\end{eqnarray}
\end{lemma}

\begin{lemma}
\label{lemma 3thm2}
Under the conditions of Theorem \ref{theorem 1} for the reference outputs, if $\tilde{\bar{v}}\in H^1(0,1)$ then $\tilde{v}\in H^1(0,1)$ and there exists a positive constant $\xi_1$ such that the following holds
\begin{eqnarray}
%\|\tilde{u}(t)\|_{H^1}\leq \alpha_3\left(\|\tilde{v}(t)\|_{H^1}\right).\label{lemrel1n}
\|\tilde{v}(t)\|_{H^1}\leq \xi_1\|\tilde{\bar{v}}(t)\|_{H^1}.\label{lemrel1nne}
\end{eqnarray}
In reverse, if $\tilde{{v}}\in H^1(0,1)$ then $\tilde{\bar{v}}\in H^1(0,1)$ and there exists a positive constant $\xi_2$ such that the following holds
\begin{eqnarray}
%\|\tilde{u}(t)\|_{H^1}\leq \alpha_3\left(\|\tilde{v}(t)\|_{H^1}\right).\label{lemrel1n}
\|\tilde{\bar{v}}(t)\|_{H^1}\leq \xi_2\|\tilde{{v}}(t)\|_{H^1}.\label{lemrel1nne1}
\end{eqnarray}
\end{lemma}

%We are now ready to prove Theorem \ref{theorem 2}.

%remark one could stabilize without full state feedback but no arbirtrail rate of covnergence. note the open-loop has am arbitray constant an equilibrium

\begin{proof}[Proof of Theorem \ref{theorem 2}]
We start by considering the following backstepping transformation, which is introduced in \cite{Vazquez1} (see also Fig. \ref{inter})
\begin{eqnarray}
w(x,t)=\tilde{v}(x,t)-\int_{-x+1}^xk\left(x,\xi\right)\tilde{v}\left(\xi,t\right)d\xi,\label{trans1}
\end{eqnarray}
where $k$ is defined in (\ref{poo}).
%\begin{eqnarray}
%k(x,\xi)&=&-\frac{1}{2}\sqrt{\frac{c_1}{\epsilon}}\frac{{\rm I}_{1}\left(\sqrt{\frac{c_1}{\epsilon}\left(\left(x-\frac{1}{2}\right)^2-\left(\xi-\frac{1}{2}\right)^2\right)}\right)}{\sqrt{\left(x-\frac{1}{2}\right)^2-\left(\xi-\frac{1}{2}\right)^2}}\left(x+\xi-1\right)\label{poo}.
%\left\{\begin{array}{ll}-r(x,\xi),&\mbox{for all $0\leq x\leq\frac{1}{2}$}\\r(x,\xi),&\mbox{for all $\frac{1}{2}\leq x\leq 1$}\end{array}\right.\label{poo}.
%-\frac{1}{2}\sqrt{\frac{c_1}{\epsilon}}{\rm sgn} \left(x-\frac{1}{2}\right){\rm I}_{1}\left(\sqrt{\frac{c_1}{\epsilon}\left(\left(x-\frac{1}{2}\right)^2-\left(\xi-\frac{1}{2}\right)^2\right)}\right)\sqrt{\frac{x+\xi-1}{x-\xi}}\\
%r(x,\xi)&=&-\frac{1}{2}\sqrt{\frac{c_1}{\epsilon}}{\rm I}_{1}\left(\sqrt{\frac{c_1}{\epsilon}\left(\left(x-\frac{1}{2}\right)^2-\left(\xi-\frac{1}{2}\right)^2\right)}\right)\sqrt{\frac{x+\xi-1}{x-\xi}},
%\nonumber\\
%&=&k(x,\xi),K\left(x-\frac{1}{2},\xi-\frac{1}{2}\right)
%\end{eqnarray}
Transformation (\ref{trans1}), together with the control laws (\ref{bac1}), (\ref{bac2}), map system (\ref{error1})--(\ref{error2}) to \cite{Vazquez1}\footnote{In order to become clear how the results from \cite{Vazquez1} are employed we provide additional details in Appendix C.} 
\begin{eqnarray}
w_t(x,t)&=&\epsilon {w}_{xx}(x,t)-\left(\frac{a^2b^2}{4\epsilon}+c_1\right)w(x,t)\label{sys0ervw}\\
w_x(0,t)&=&0\label{bgh}\\
w_x(1,t)&=&0,\label{sysnervw}
\end{eqnarray}
where $c_1>0$ is arbitrary. The backstepping transformation (\ref{trans1}) is invertible with inverse that may be expressed as
\begin{eqnarray}
\tilde{v}(x,t)={w}(x,t)+\int_{-x+1}^xl\left(x,\xi\right){w}\left(\xi,t\right)d\xi,\label{trans1zinvx}
\end{eqnarray}
which follows specializing the results in \cite{vazquez2} to the present case\footnote{To see this, note that a one-dimensional ball is, in fact, an interval and, its boundary, i.e., a ``zero-sphere" just consists of the two endpoints of the interval.}(see also the discussion in \cite{Vazquez1}), where\footnote{In terms of the $z=x-\frac{1}{2}$ and $y=\xi-\frac{1}{2}$ variables the inverse backstepping transformation (\ref{trans1zinvx}) can be written as $\tilde{v}_1(z,t)={w}_1(z,t)+\int_{-z}^zL\left(z,y\right){w}_1\left(y,t\right)dy$, where $w_1(z,t)=w\left(z+\frac{1}{2},t\right)$, $\tilde{v}_1(z,t)=\tilde{v}\left(z+\frac{1}{2},t\right)$, and $L(z,y)=l\left(z+\frac{1}{2},y+\frac{1}{2}\right)$.}
%\nonumber\\
%&=&k(x,\xi),K\left(x-\frac{1}{2},\xi-\frac{1}{2}\right)
\begin{eqnarray}
%k(x,\xi)&=&-\frac{1}{2}\sqrt{\frac{c_1}{\epsilon}}{\rm sgn} \left(x-\frac{1}{2}\right){\rm I}_{1}\left(\sqrt{\frac{c_1}{\epsilon}\left(\left(x-\frac{1}{2}\right)^2-\left(\xi-\frac{1}{2}\right)^2\right)}\right)\sqrt{\frac{x+\xi-1}{x-\xi}}\\
l(x,\xi)&=&-\frac{1}{2}\sqrt{\frac{c_1}{\epsilon}}\frac{{\rm J}_{1}\left(\sqrt{\frac{c_1}{\epsilon}\left(\left(x-\frac{1}{2}\right)^2-\left(\xi-\frac{1}{2}\right)^2\right)}\right)}{\sqrt{\left(x-\frac{1}{2}\right)^2-\left(\xi-\frac{1}{2}\right)^2}}\nonumber\\
&&\times\left(x+\xi-1\right)\label{laz},
%\bar{r}\left(x,\xi\right)&=&-\frac{1}{2}\sqrt{\frac{c_1}{\epsilon}}{\rm J}_{1}\left(\sqrt{\frac{c_1}{\epsilon}\left(\left(x-\frac{1}{2}\right)^2-\left(\xi-\frac{1}{2}\right)^2\right)}\right)\sqrt{\frac{x+\xi-1}{x-\xi}},\label{bar r}
%\nonumber\\
%&=&k(x,\xi),K\left(x-\frac{1}{2},\xi-\frac{1}{2}\right)
\end{eqnarray}
with ${\rm J_1}$ being the first-order Bessel function of the first kind. Having defined the backstepping transformation and its inverse it is shown, specializing the results in \cite{vazquez2} (Section 6.3), that there exist positive constants $m_1$ and $m_2$ such that%\footnote{An alternative and, perhaps, more direct way to see this is provided in Appendix E.}
\begin{eqnarray}
\|w(t)\|_{H^1}&\leq& m_1\|\tilde{v}(t)\|_{H^1}\label{bound dir}\\
\|\tilde{v}(t)\|_{H^1}&\leq& m_2 \|w(t)\|_{H^1}.\label{bound in}
\end{eqnarray}
Defining the  Lyapunov functional 
\begin{eqnarray}
S_1(t)=\frac{1}{2}\int_0^1w(x,t)^2dx+\frac{1}{2}\int_0^1w_x(x,t)^2dx,
\end{eqnarray}
we get along the solutions of the ``target" system (\ref{sys0ervw})--(\ref{sysnervw}) that $\dot{S}_1(t)\leq -2\left(c_1+\frac{a^2b^2}{4\epsilon}\right) S_1(t)$, where we took the $L^2$-inner product of (\ref{sys0ervw}) with $w$, $w_{xx}$ and performed one step of integration by parts. Using (\ref{bound dir}), (\ref{bound in}), we get for all $t\geq t_0$
\begin{eqnarray}
\|\tilde{v}(t)\|_{H^1}&\leq& m_1m_2\sqrt{2} \|\tilde{v}\left(t_0\right)\|_{H^1}e^{ -\left(c_1+\frac{a^2b^2}{4\epsilon}\right)\left(t-t_0\right)},\label{bound0}
\end{eqnarray}
and hence, from Lemma \ref{lemma 3thm2}, we conclude that there exists a constant $\nu_4$ such that for all $t\geq t_0$
\begin{eqnarray}
\|\tilde{\bar{v}}(t)\|_{H^1}&\leq&\nu_4 \|\tilde{\bar{v}}\left(t_0\right)\|_{H^1}e^{ -\left(c_1+\frac{a^2b^2}{4\epsilon}\right)\left(t-t_0\right)}.\label{bound00}
\end{eqnarray}
From Lemma \ref{lemma 1thm2}, estimate (\ref{bound00}) implies that one can choose $\mu$ in (\ref{region}) sufficiently small, in fact, such that $\mu\leq\alpha_1^{-1}\left(\frac{c}{2\nu_4}\right)$, in order for relation $\|\tilde{\bar{v}}(t)\|_{H^1}<\frac{c}{2}$, for some $0<c<1$, to hold for all $t\geq t_0$. Hence, since $\sup_{x \in[0,1]}|\tilde{\bar{v}}(x,t)|\leq 2\|\tilde{\bar{v}}(t)\|_{H^1}$, for any $\tilde{v}\in H^1(0,1)$, we conclude that condition (\ref{feasibility}), for some $0<c<1$, is satisfied. Estimate (\ref{thm1est}) is then obtained, employing Lemma \ref{lemma 2thm2} and combining estimate (\ref{bound00}) with estimates (\ref{lemrel1}), (\ref{lemrel1n}), with $\beta\left(s,t-t_0\right)= \frac{\nu_4\epsilon}{|a|\left(1-c\right)}\alpha_1\left(s\right)e^{ -\left(c_1+\frac{a^2b^2}{4\epsilon}\right)\left(t-t_0\right)}$.

We study next the well-posedness of the closed-loop system. We start with the target system (\ref{sys0ervw})--(\ref{sysnervw}). Since from (\ref{poo}) we get that $k\in C^2\left(D\right)$, from transformation (\ref{trans1zinvx}) it follows, using the fact that $\tilde{v}_0\in H^2(0,1)$ (which follows, in a similar way to the derivation of estimates (\ref{lemrel1}), (\ref{lemrel1nne}), from (\ref{transformation1er}), (\ref{transformation1add}) exploiting the regularity assumption on $u_0$ and the regularity properties of the reference trajectory $u^{\rm r}\left(t_0\right)$ in Theorem \ref{theorem 1}) and the compatibility conditions, that $w_0\in H^2(0,1)$ satisfies the compatibility conditions $w'_0(0)=w'_0(1)=0$. Therefore, from (\ref{sys0ervw})--(\ref{sysnervw}) it is shown, see, for example, \cite{brezis}, that there exists a unique ${w}\in C\left([t_0,\infty);H^2(0,1)\right)$. The inverse transformation (\ref{trans1zinvx}) and the fact that $l\in C^2\left(D\right)$ (which follows from expression (\ref{laz})) guarantee the existence and uniqueness of $\tilde{v}\in C\left([t_0,\infty);H^2(0,1)\right)$. Using (\ref{transformation1er}), (\ref{transformation1add}) it follows that $\tilde{u}(x,t)=-\frac{\epsilon}{a}\textrm{log}\left(\tilde{{v}}(x,t)e^{\frac{ab}{2\epsilon}+\frac{a}{\epsilon}u^{\rm r}(x,t)}+1\right)$, and hence, in a similar way to the derivation of estimates (\ref{lemrel1n}), (\ref{lemrel1nne1}),  the regularity properties of $u^{\rm r}$ and condition (\ref{feasibility}) guarantee the existence and uniqueness of ${u}\in C\left([t_0,\infty);H^2(0,1)\right)$. Employing similar arguments, with \cite{brezis} (see also, e.g., \cite{lasi}) it is shown that $u\in C^{2,1}\left([0,1]\times \left(t_0,\infty\right)\right)$. \qed
\end{proof}

\section{Nonlinear Observer and Output Feedback Law Designs}
\label{sec observer}

In this section, we design a nonlinear observer to estimate the state $\tilde{u}(x)$, $x\in[0,1]$, which may be employed in the full-state feedback laws (\ref{con1}), (\ref{con2}) giving rise to an observer-based output-feedback design or, it may be utilized independently when the goal is only state estimation. The observer utilizes measurements from both ends of the spatial domain. Furthermore, we also present static, collocated output-feedback controllers, which, however, cannot achieve an arbitrary decay rate.

%OBSRVER IF THE ARBITRAY RATE IS EMOPLOYED OR ONLY FOR ESTIMATIENG THE STATE USUFEUL IN APPLICATIONS

%can observer be decentralized? i.e., at zero the controller uses observer with measurement at zero and same ofr the other end.

\subsection{Observer design}
Exploiting the convenient form of system (\ref{error1})--(\ref{error2}) we introduce the following observer 
\begin{eqnarray}
\hat{\tilde{v}}_t(x,t)&=&\epsilon \hat{\tilde{v}}_{xx}(x,t)-\frac{a^2b^2}{4\epsilon}\hat{\tilde{v}}(x,t)+p_2(x)\nonumber\\
&&\times \left(\left(e^{-\frac{a}{\epsilon}\tilde{u}(0,t)}-1\right)e^{-\frac{a}{\epsilon}u^{\rm r}(0,t)}-\hat{\tilde{v}}(0,t)\right)\nonumber\\
&&+p_1(x)\left(\left(e^{-\frac{a}{\epsilon}\tilde{u}(1,t)}-1\right)\right.\nonumber\\
&&\left.\times e^{-\frac{ab}{2\epsilon}-\frac{a}{\epsilon}u^{\rm r}(1,t)}-\hat{\tilde{v}}(1,t)\vphantom{\left(e^{-\frac{a}{\epsilon}\tilde{u}(1,t)}-1\right)}\right)\label{obs1}\\
\hat{\tilde{v}}_x(0,t)&=&\tilde{V}_0(t)+p_{00}\left(\left(e^{-\frac{a}{\epsilon}\tilde{u}(0,t)}-1\right)e^{-\frac{a}{\epsilon}u^{\rm r}(0,t)}\right.\nonumber\\
&&\left.-\hat{\tilde{v}}(0,t)\right)\\
\hat{\tilde{v}}_x(1,t)&=&\tilde{V}_1(t)+p_{11}\left(\left(e^{-\frac{a}{\epsilon}\tilde{u}(1,t)}-1\right)e^{-\frac{ab}{2\epsilon}-\frac{a}{\epsilon}u^{\rm r}(1,t)}\right.\nonumber\\
&&\left.-\hat{\tilde{v}}(1,t)\right).\label{obs2}
\end{eqnarray}
The gains $p_2(x)$, $p_1(x)$, $p_{00}$, and $p_{11}$ are designed via the backstepping methodology, specializing the results from \cite{vazquez2} to a one-dimensional spatial domain, as
%\begin{eqnarray}
%p_0(x)&=&-\epsilon P_{\xi}\left(x-\frac{1}{2},-\frac{1}{2}\right)\\
%p_1(x)&=&-\epsilon P_{\xi}\left(x-\frac{1}{2},\frac{1}{2}\right)\\
%p_{00}&=&-P\left(-\frac{1}{2},-\frac{1}{2}\right)\\
%p_{01}&=&P\left(-\frac{1}{2},\frac{1}{2}\right)\\
%p_{10}&=&P\left(\frac{1}{2},-\frac{1}{2}\right)\\
%p_{11}&=&-P\left(\frac{1}{2},\frac{1}{2}\right),
%\end{eqnarray}
\begin{eqnarray}
p_2(x)&=&-\epsilon P_{\xi}\left(x,0\right)\label{gain1}\\
p_1(x)&=&-\epsilon P_{\xi}\left(x,1\right)\label{gainp2}\\
p_{00}&=&-P\left(0,0\right)\\
p_{11}&=&-P\left(1,1\right),\label{gain2}
\end{eqnarray}
where the kernel $P$ is given explicitly, for $(x,\xi)$ in the domain $E=E_1\cup 	E_2$, where $E_1=\left\{\left(x,\xi\right):\frac{1}{2}\leq \xi\leq 1,\right.$ $\left.-\xi+1\leq x\leq \xi\right\}$ and $E_2=\left\{\left(x,\xi\right):0\leq \xi\leq \frac{1}{2},\right.\hphantom{dssdsdsd}$ $\left.\xi\leq x\leq 1-\xi\right\}$, by
\begin{eqnarray}
%R\left(x,\xi\right)=-\frac{1}{2}\sqrt{\frac{c_2}{\epsilon}}{\rm I}_{1}\left(\sqrt{\frac{c_2}{\epsilon}\left(\left(\xi-\frac{1}{2}\right)^2-\left(x-\frac{1}{2}\right)^2\right)}\right)\sqrt{\frac{\xi+x-1}{\xi-x}},\label{gainR}\\
%-\frac{1}{2}\sqrt{\frac{c_2}{\epsilon}}\frac{{\rm I}_{1}\left(\sqrt{\frac{c_2}{\epsilon}\left(y^2-z^2\right)}\right)}{\sqrt{y^2-z^2}}\left(y+z\right)\\
P\left(x,\xi\right)&=&-\frac{1}{2}\sqrt{\frac{c_2}{\epsilon}}\frac{{\rm I}_{1}\left(\sqrt{\frac{c_2}{\epsilon}\left(\left(\xi-\frac{1}{2}\right)^2-\left(x-\frac{1}{2}\right)^2\right)}\right)}{\sqrt{\left(\xi-\frac{1}{2}\right)^2-\left(x-\frac{1}{2}\right)^2}}\nonumber\\
&&\times\left(\xi+x-1\right)\label{gainR},
\end{eqnarray}
%where the kernel $P$ is given explicitly by
%\begin{eqnarray}
%P\left(x-\frac{1}{2},\xi-\frac{1}{2}\right)=-\frac{1}{2}\sqrt{\frac{c_2}{\epsilon}}{\rm sgn} \left(\xi-\frac{1}{2}\right){\rm I}_{1}\left(\sqrt{\frac{c_2}{\epsilon}\left(\left(\xi-\frac{1}{2}\right)^2-\left(x-\frac{1}{2}\right)^2\right)}\right)\sqrt{\frac{\xi+x-1}{\xi-x}},
%\end{eqnarray}
where ${\rm I}_1$ denotes the modified Bessel function of the first kind of first order and $c_2>0$ is arbitrary. 

Note that observer (\ref{obs1})--(\ref{obs2}) is a copy of the (linear) system (\ref{error1})--(\ref{error2}) plus output injection, where the output-injection terms are linear in the state $\tilde{v}$, which can be seen using relations (\ref{transformation1er}), (\ref{transformation1add}) for $x=0$ and $x=1$.

%The form of the observer (\ref{obs1})--(\ref{obs2}) is motivated as follows. Using relations (\ref{transformation1er}), (\ref{transformation1add}) for $x=0$ and $x=1$ we get 
%Employing the invertible transformation $e(x,t)=z(x,t)-\int_{-x+1}^{x}p(x,y)z(y,t)dy$, we map $e$ system to 

%$z_t(x,t)=\epsilon {z}_{xx}(x,t)$ $-\left(\frac{1}{4\epsilon}+c_2\right)z(x,t)$, $z_x(0,t)\!=\!z_x(1,t)\!=\!0$, if $p_1$, $p_{10}$ satisfy

%\begin{eqnarray}
%p_1(x)&=&-\epsilon p_y(x,1)\label{99}\\
%p_{10}&=&p(1,1),\label{1000}
%\end{eqnarray}
%where $p$ satisfies $p_{xx}(x,y)-p_{yy}(x,y)=-\frac{c_2}{\epsilon}p(x,y)$, $\frac{dp(x,x)}{dx}=-\frac{c_2}{2\epsilon}$, $p_x(0,y)=0$, with $p(0,0)=0$, such that $z_x(0,t)=0$ holds given $e_x(0,t)=0$. Existence and uniqueness of $p\in C^2(B)$, where $B=\left\{(x,y):0\leq x\leq y\leq 1\right\}$, follows \cite{andrey observer}. In fact, $p(x,y)=-\frac{c_2}{\epsilon}y\frac{I_1\left(\sqrt{\frac{c_2}{\epsilon}\left(y^2-x^2\right)}\right)}{\sqrt{\frac{c_2}{\epsilon}\left(y^2-x^2\right)}}$ \cite{andrey observer}. %Comparing (\ref{Px1})--(\ref{Px2}) to (\ref{kx1})--(\ref{kx2}) one can conclude that 

\subsection{Observer-based output feedback boundary control design}

In order to employ the full-state feedback laws (\ref{con1}), (\ref{con2}), which may achieve an arbitrary decay rate for the closed-loop system, utilizing only boundary measurements, we first modify the control laws (\ref{bac1}), (\ref{bac2}) as
\begin{eqnarray}
\tilde{V}_0(t)&=&k\left(0,0\right)\tilde{v}\left(0,t\right)-\int_{0}^1k_x\left(0,\xi\right)\hat{\tilde{v}}\left(\xi,t\right)d\xi\label{bac1ob}\\
\tilde{V}_1(t)&=&k\left(1,1\right)\tilde{v}\left(1,t\right)+\int_{0}^1k_x\left(1,\xi\right)\hat{\tilde{v}}\left(\xi,t\right)d\xi,\label{bac2ob}
%\tilde{V}_0(t)&=&k\left(0,0\right)\tilde{v}\left(0,t\right)-\int_{0}^1k_x\left(0,\xi\right)\tilde{v}\left(\xi,t\right)d\xi\label{bac1}\\
%\tilde{V}_1(t)&=&k\left(1,1\right)\tilde{v}\left(1,t\right)+\int_{0}^1k_x\left(1,\xi\right)\tilde{v}\left(\xi,t\right)d\xi,\label{bac2}
\end{eqnarray}
and hence, the control laws (\ref{con1}), (\ref{con2}) now become
\begin{eqnarray}
{U}_0(t)&=&-\frac{\epsilon}{a} e^{\frac{a}{\epsilon}{\tilde{u}}(0,t)}\left(\left(k(0,0)+\frac{ab}{2\epsilon}\right)\left(e^{-\frac{a}{\epsilon}\tilde{u}(0,t)}-1\right)\right.\nonumber\\
&&\left.- e^{\frac{a}{\epsilon}{u^{\rm r}}(0,t)}\int_{0}^1k_x\left(0,\xi\right)\hat{\tilde{v}}\left(\xi,t\right)d\xi\vphantom{\left(k(0,0)+\frac{ab}{2\epsilon}\right)\left(e^{-\frac{a}{\epsilon}\tilde{u}(0,t)}-1\right)}\right)\nonumber\\
&&+U_0^{\rm r}(t)e^{\frac{a}{\epsilon}{\tilde{u}}(0,t)}\label{con1obs}\\
{U}_1(t)&=&-\frac{\epsilon}{a} e^{\frac{a}{\epsilon}{\tilde{u}}(1,t)}\left(\left(k\left(1,1\right)+\frac{ab}{2\epsilon}\right)\left(e^{-\frac{a}{\epsilon}\tilde{u}(1,t)}-1\right)\right.\nonumber\\
&&\left.+e^{\frac{ab}{2\epsilon}+\frac{a}{\epsilon}u^{\rm r}\left(1,t\right)}\int_{0}^1k_x\left(1,\xi\right)\hat{\tilde{v}}\left(\xi,t\right)d\xi\right)\nonumber\\
&&+U_1^{\rm r}(t)e^{\frac{a}{\epsilon}{\tilde{u}}(1,t)}.\label{con2obs}
\end{eqnarray}

\subsection{Static collocated output-feedback controllers}
Provided that $b\neq0$, the zero solution of system (\ref{error1})--(\ref{error2}) is asymptotically stable when $\tilde{V}_0(t)=\tilde{V}_1(t)=0$ for all $t\geq t_0$ (and hence, so is the zero solution of system (\ref{sys0erver})--(\ref{sysnerver}) provided that $u^{\rm r}$ is uniformly bounded), i.e., when 
\begin{eqnarray}
{U}_0(t)&=&\frac{b}{2}\left(e^{\frac{a}{\epsilon}\tilde{u}(0,t)}-1\right)+U_0^{\rm r}(t)e^{\frac{a}{\epsilon}\tilde{u}(0,t)}\label{static1}\\
{U}_1(t)&=&\frac{b}{2}\left(e^{\frac{a}{\epsilon}\tilde{u}(1,t)}-1\right)+U_1^{\rm r}(t)e^{\frac{a}{\epsilon}\tilde{u}(1,t)}, \label{static2}
\end{eqnarray}
which may be viewed as decentralized (in the sense that each controller requires measurements of the state at the same boundary), static output-feedback control laws. However, the convergence rate to the zero equilibrium of the closed-loop solution of system (\ref{error1})--(\ref{error2}) under the control laws (\ref{static1}), (\ref{static2}) is not arbitrary (in contrast to the achievable decay rate under the full-state feedback laws (\ref{con1}), (\ref{con2}), which is arbitrary), but depends on the parameters of the system, namely $a$, $b$, and $\epsilon$.

\section{Trajectory Tracking Under Observer-Based Output Feedback}
\label{sec observer sta}

We next state and prove the following stability result for the closed-loop system, under the observer-based output feedback law.% Its proof can be found in Appendix E.
\begin{theorem}
\label{thm3}
Consider a closed-loop system consisting of system (\ref{sys1})--(\ref{sys3}), the control laws (\ref{con1obs}), (\ref{con2obs}), and the observer (\ref{obs1})--(\ref{obs2}) with (\ref{bac1ob}), (\ref{bac2ob}). Under the conditions of Theorem \ref{theorem 1} for the reference outputs, there exist a positive constant $\mu^*$ and a class $\mathcal{KL}$ function $\beta^*$ such that for all initial conditions $\left({u}_0,\hat{\tilde{v}}_0\right)\in H^2\left(0,1\right)\times H^2\left(0,1\right)$ which are compatible with the control laws (\ref{bac1ob})--(\ref{con2obs}) and which satisfy 
\begin{eqnarray}
\|\tilde{u}\left(t_0\right)\|_{H^1}+\|\hat{\tilde{v}}\left(t_0\right)\|_{H^1}&<&\mu^*\label{region1},
\end{eqnarray}
the following holds
\begin{eqnarray}
\Omega(t)&\leq&\beta^*\left(\Omega\left(t_0\right),t-t_0\right),\quad\mbox{for all $t\geq t_0$}\label{tg}\\
\Omega(t)&=&\|\tilde{u}\left(t\right)\|_{H^1}+\|\hat{\tilde{v}}\left(t\right)\|_{H^1}.
\end{eqnarray}
Moreover, the closed-loop system has a unique solution $u,\hat{\tilde{v}}\in C\left([t_0,\infty);H^2(0,1)\right)$ with $u,\hat{\tilde{v}}\in C^{2,1}\left([0,1]\times \left(t_0,\infty\right)\right)$.
\end{theorem}

\begin{proof}[Proof of Theorem \ref{thm3}] The proof is divided into three parts. 

{\textbf{Part 1}}: {\em Backstepping transformation of the state estimation error}

%We first show that system (\ref{esys1})--(\ref{esys2}), with the gains (\ref{gain1})--(\ref{gainR}), is exponentially stable in the $H^1$ norm, specializing the results from \cite{vazquez2} (Section 5). 

We start defining the state estimation error 
\begin{eqnarray}
e=\tilde{v}-\hat{\tilde{v}}.\label{deferror}
\end{eqnarray}
Using relations (\ref{error1})--(\ref{error2}) for the $\tilde{v}$ system and relations (\ref{obs1})--(\ref{obs2}) for the observer, we get with equations (\ref{transformation1er}), (\ref{transformation1add}) for $x=0$ and $x=1$ that the state estimation error $e$ satisfies the PDE $e_t(x,t)=\epsilon {e}_{xx}(x,t)-\frac{a^2b^2}{4\epsilon}e(x,t)-p_2(x)e(0,t)-p_1(x)e(1,t)$ with boundary conditions $e_x(0,t)=-p_{00}e(0,t)$ and $e_x(1,t)=-p_{11}e(1,t)$. Since it turns out to be convenient to shift from the variable $x$ to the variable $z=x-\frac{1}{2}$ (in order to make the connection with the results from \cite{vazquez2} more clear), we re-write the error system as
\begin{eqnarray}
\bar{e}_t(z,t)&=&\epsilon \bar{e}_{zz}(z,t)-\frac{a^2b^2}{4\epsilon}\bar{e}(z,t)\nonumber\\
&&-\bar{p}_2(z)\bar{e}\left(-\frac{1}{2},t\right)-\bar{p}_1(z)\bar{e}\left(\frac{1}{2},t\right)\label{esys1z}\\
\bar{e}_z\left(-\frac{1}{2},t\right)&=&-p_{00}\bar{e}\left(-\frac{1}{2},t\right)\label{e sys lot}\\
\bar{e}_z\left(\frac{1}{2},t\right)&=&-p_{11}\bar{e}\left(\frac{1}{2},t\right),\label{esys2z}
\end{eqnarray}
where we define $\bar{e}(z,t)=e\left(z+\frac{1}{2}\right)$ as well as $\bar{p}_2(z)=p_2\left(z+\frac{1}{2}\right)$ and $\bar{p}_1(z)=p_1\left(z+\frac{1}{2}\right)$. Consider the following transformation, which is derived specializing the result from \cite{vazquez2} (Section 5) to the case of a one-dimensional spatial domain
\begin{eqnarray}
\bar{e}(z,t)&=&\bar{w}(z,t)-\int_z^{\frac{1}{2}}p\left(z,y\right)\bar{w}\left(y,t\right)dy\nonumber\\
&&+\int_{-\frac{1}{2}}^{-z}p\left(z,y\right)\bar{w}\left(y,t\right)dy,\quad 0\leq z\leq \frac{1}{2}\label{r1}\\
\bar{e}(z,t)&=&\bar{w}(z,t)+\int_{-\frac{1}{2}}^{z}p\left(z,y\right)\bar{w}\left(y,t\right)dy\nonumber\\
&&-\int_{-z}^{\frac{1}{2}}p\left(z,y\right)\bar{w}\left(y,t\right)dy,\quad -\frac{1}{2}\leq z\leq 0,\label{r2}
\end{eqnarray}
where
\begin{eqnarray}
p\left(z,y\right)=P\left(z+\frac{1}{2},y+\frac{1}{2}\right)\label{pp},
%\left\{\begin{array}{ll}-R\left(z+\frac{1}{2},y+\frac{1}{2}\right),&\mbox{for $-\frac{1}{2}\leq y\leq 0$}\\R\left(z+\frac{1}{2},y+\frac{1}{2}\right),&\mbox{for $0\leq y\leq \frac{1}{2}$}\end{array}\right.,\label{pp}
\end{eqnarray}
and the kernel $P$ is defined in (\ref{gainR}). From \cite{vazquez2} (Section 5) it follows\footnote{For the reader's benefit, we provide some further explanations in Appendix D, which are given, specifically, for the case of a one-dimensional spatial domain.} that transformation (\ref{r1}), (\ref{r2}) maps the following system into (\ref{esys1z})--(\ref{esys2z}) 
\begin{eqnarray}
\bar{w}_t(z,t)&=&\epsilon \bar{w}_{zz}(z,t)-\left(\frac{a^2b^2}{4\epsilon}+c_2\right)\bar{w}(z,t)\label{z1}\\
\bar{w}_z\left(-\frac{1}{2},t\right)&=&0\label{zn lot}\\
\bar{w}_z\left(\frac{1}{2},t\right)&=&0.\label{zn}
\end{eqnarray}
%Moreover, transformation (\ref{r1}), (\ref{r2}) is invertible and its inverse may be expressed as \cite{vazquez2}\footnote{Although this follows specializing the results from \cite{vazquez2}, in Appendix G we provide an alternative explanation, specifically for transformation (\ref{r1}), (\ref{r2}).}
Moreover, transformation (\ref{r1}), (\ref{r2}) is invertible and its inverse may be expressed, specializing the results from \cite{vazquez2}, as %\cite{vazquez2}\footnote{Although this follows specializing the results from \cite{vazquez2}, in Appendix G we provide an alternative explanation, specifically for transformation (\ref{r1}), (\ref{r2}).}
\begin{eqnarray}
\bar{w}(z,t)&=&\bar{e}(z,t)+\int_z^{\frac{1}{2}}\bar{p}\left(z,y\right)\bar{e}\left(y,t\right)dy\nonumber\\
&&-\int_{-\frac{1}{2}}^{-z}\bar{p}\left(z,y\right)\bar{e}\left(y,t\right)dy,\quad 0\leq z\leq \frac{1}{2}\label{r1in}\\
\bar{w}(z,t)&=&\bar{e}(z,t)-\int_{-\frac{1}{2}}^{z}\bar{p}\left(z,y\right)\bar{e}\left(y,t\right)dy\nonumber\\
&&+\int_{-z}^{\frac{1}{2}}\bar{p}\left(z,y\right)\bar{e}\left(y,t\right)dy,\quad -\frac{1}{2}\leq z\leq 0,\label{r2in}
\end{eqnarray}
where the kernel $\bar{p}(z,y)$ has a very similar structure to $l\left(z+\frac{1}{2},y+\frac{1}{2}\right)$ in (\ref{laz}). 
%\begin{eqnarray}
%\bar{p}\left(z,y\right)=\left\{\begin{array}{ll}-\bar{R}\left(z,y\right),&\mbox{for $-\frac{1}{2}\leq y\leq 0$}\\\bar{R}\left(z,y\right),&\mbox{for $0\leq y\leq \frac{1}{2}$}\end{array}\right.,\label{ppin}
%\end{eqnarray}
%and the kernel $\bar{R}(z,y)$ has a very similar structure to $l\left(z+\frac{1}{2},y+\frac{1}{2}\right)$ in (\ref{laz}). 

%Having defined the direct and inverse backstepping transformations for the state estimation error, it can be shown, either utilizing the result from \cite{vazquez2} (Section 6) or employing almost identical arguments to the proofs of estimates (\ref{bound dir}), (\ref{bound in}) in the proof of Theorem \ref{theorem 2} (see Appendix E), that there exist positive constants $m_3$ and $m_4$ such that 

Having defined the direct and inverse backstepping transformations for the state estimation error, it can be shown, utilizing the results from \cite{vazquez2} (Section 6, where almost identical arguments to the proofs of estimates (\ref{bound dir}), (\ref{bound in}) in the proof of Theorem \ref{theorem 2} are employed), that there exist positive constants $m_3$ and $m_4$ such that 
\begin{eqnarray}
\|\bar{w}(t)\|_{H^1}&\leq& m_3\|\bar{e}(t)\|_{H^1}\label{bound dirobs}\\
\|\bar{e}(t)\|_{H^1}&\leq& m_4 \|\bar{w}(t)\|_{H^1}.\label{bound inobs}
\end{eqnarray}

{\textbf{Part 2}}: {\em Backstepping transformation of the observer state}

%We next show that the $\left(\bar{w},\hat{\tilde{v}}\right)$ system (\ref{esys1})--(\ref{esys2}), (\ref{obs1})--(\ref{obs2}) with (\ref{bac1ob}), (\ref{bac2ob}) is exponentially stable in the $H^1$ norm. Toward this, we specialize the arguments from \cite{vazquez2} (Section 6) to the case of . Define the transformation 
%\begin{eqnarray}
%e(x,t)=z(x,t)-\int_{-x+1}^{x}p\left(x,\xi\right)z\left(\xi,t\right)d\xi, \label{ztr}
%\end{eqnarray}
%where $p\left(x,\xi\right)=P\left(x-\frac{1}{2},\xi-\frac{1}{2}\right)$. It can be shown that transformation (\ref{ztr}) is invertible and its inverse satisfies the following system
%\begin{eqnarray}
%z_t(x,t)&=&\epsilon {z}_{xx}(x,t)-\left(\frac{a^2b^2}{4\epsilon}+c_2\right)z(x,t)\label{z1}\\
%z_x(0,t)&=&0\\
%z_x(1,t)&=&0.\label{zn}
%\end{eqnarray}
Consider the transformation 
\begin{eqnarray}
\hat{w}_1(z,t)=\hat{\tilde{v}}_1(z,t)-\int_{-z}^zK\left(z,y\right)\hat{\tilde{v}}_1\left(y,t\right)dy,\label{trans1ob}
\end{eqnarray}
where $\hat{\tilde{v}}_1(z,t)=\hat{\tilde{v}}\left(z+\frac{1}{2},t\right)$ and $K(z,y)=k\left(z+\frac{1}{2},y+\frac{1}{2}\right)$, with $k$ being defined in (\ref{poo}). Following \cite{vazquez2} (see also the discussion in \cite{Vazquez1}), it is shown that the inverse of transformation (\ref{trans1ob}) is defined, similarly to the case of transformation (\ref{trans1}), as
\begin{eqnarray}
\hat{\tilde{v}}_1(z,t)=\hat{w}_1(z,t)+\int_{-z}^zL\left(z,y\right)\hat{w}_1\left(y,t\right)dy,\label{trans1zinvzzob}
\end{eqnarray} 
where $L(z,y)=l\left(z+\frac{1}{2},y+\frac{1}{2}\right)$, with $l$ being given in (\ref{laz}). Noting that the variable $\hat{\tilde{v}}_1$ satisfies the same PDE system with the variable $\hat{\tilde{v}}$, i.e., system (\ref{obs1})--(\ref{obs2}), with the difference that the variable $x$ is shifted to $z=x-\frac{1}{2}$, one can conclude that transformation (\ref{trans1ob}) together with the control laws (\ref{bac1ob}), (\ref{bac2ob}) map the $\hat{\tilde{v}}_1$ system to
\begin{eqnarray}
\hat{w}_{1_t}(z,t)&=&\epsilon \hat{w}_{1_{zz}}(z,t)-\left(\frac{a^2b^2}{4\epsilon}+c_1\right)\hat{w}_1(z,t)\nonumber\\
&&+\left(\bar{p}_2(z)-\int_{-z}^zK\left(z,y\right)\bar{p}_2\left(y\right)dy\right)\nonumber\\
&&\times\bar{w}\left(-\frac{1}{2},t\right)+\left(\vphantom{\int_{-z}^zK\left(z,y\right)\bar{p}_1\left(y\right)dy}\bar{p}_1(z)\right.\nonumber\\
&&\left.-\int_{-z}^zK\left(z,y\right)\bar{p}_1\left(y\right)dy\right)\bar{w}\left(\frac{1}{2},t\right)\label{sys0ervwhat}\\
\hat{w}_{1_z}\left(-\frac{1}{2},t\right)&=&\left(k(0,0)+p_{00}\right)\bar{w}\left(-\frac{1}{2},t\right)\label{bghhat}\\
\hat{w}_{1_z}\left(\frac{1}{2},t\right)&=&\left(k(1,1)+p_{11}\right)\bar{w}\left(\frac{1}{2},t\right),\label{sysnervwhat}
\end{eqnarray}
where we also used the facts that $\bar{e}\left(\frac{1}{2},t\right)=\bar{w}\left(\frac{1}{2},t\right)$ and $\bar{e}\left(-\frac{1}{2},t\right)=\bar{w}\left(-\frac{1}{2},t\right)$, which follow from (\ref{r1}) and (\ref{r2}), respectively. From transformations (\ref{trans1ob}), (\ref{trans1zinvzzob}), employing identical arguments to the corresponding arguments within the proof of Theorem \ref{theorem 2} that led to estimates (\ref{bound dir}), (\ref{bound in}) (see also \cite{vazquez2}), it follows that there exist positive constants $m_5$ and $m_6$ such that
\begin{eqnarray}
\|\hat{w}_1(t)\|_{H^1}&\leq& m_5\|\hat{\tilde{v}}_1(t)\|_{H^1}\label{bound dirob12}\\
\|\hat{\tilde{v}}_1(t)\|_{H^1}&\leq& m_6 \|\hat{w}_1(t)\|_{H^1}.\label{bound inob13}
\end{eqnarray}

%From transformations (\ref{trans1ob}), (\ref{trans1zinvzzob}), employing identical arguments to the corresponding arguments within the proof of Theorem \ref{theorem 2} that lead to estimates (\ref{bound dir}), (\ref{bound in}), it follows that 
%\begin{eqnarray}
%\|\bar{w}(t)\|_{H^1}\leq m_1\|\hat{\tilde{v}}(t)\|_{H^1}\label{bound dirob}\\
%\|\hat{\tilde{v}}(t)\|_{H^1}\leq m_2 \|\bar{w}(t)\|_{H^1}.\label{bound inob}
%\end{eqnarray}

{\textbf{Part 3}}: {\em Stability estimates and well-posedness}

The $\left(\bar{w},\hat{w}_1\right)$ system is a cascade in which, the homogenous part of both subsystems is an exponentially stable (also in the $H^1$ norm) heat equation and the non-autonomous part, i.e., the $\hat{w}_1$ subsystem, is driven by the autonomous $\bar{w}$ subsystem. Therefore, employing similar arguments to the proof of Theorem 5 in \cite{andrey observer} (see also, e.g., \cite{deutcher}, \cite{deutcher0}, \cite{vazquez0}, \cite{vazquez2}) one can conclude that the $\left(\bar{w},\hat{w}_1\right)$ system is exponentially stable in the $H^1$ norm, and hence, so is system $\left(\bar{e},\hat{\tilde{v}}_1\right)$ (based on estimates (\ref{bound dirobs}), (\ref{bound inobs}), (\ref{bound dirob12}), and (\ref{bound inob13})). Thus,
\begin{eqnarray}
\|\hat{\tilde{v}}(t)\|_{H^1}+\|e(t)\|_{H^1}&\leq& \bar{\nu}\left(\|\hat{\tilde{v}}(t_0)\|_{H^1}+\|e(t_0)\|_{H^1}\right)\nonumber\\
&&\times e^{ -\bar{\mu} \left(t-t_0\right)}, \quad \mbox{for all $t\geq t_0$},\label{bound0023}
\end{eqnarray}
for some positive constants $\bar{\nu}$ and $\bar{\mu}$. Therefore, with definition (\ref{deferror}) and employing Lemma \ref{lemma 3thm2} we arrive at
\begin{eqnarray}
\|\hat{\tilde{v}}(t)\|_{H^1}+\|\tilde{\bar{v}}(t)\|_{H^1}&\leq& \bar{\nu}_1\left(\|\hat{\tilde{{v}}}(t_0)\|_{H^1}+\|\tilde{\bar{v}}(t_0)\|_{H^1}\right)\nonumber\\
&&\times e^{ -\bar{\mu} \left(t-t_0\right)}, \quad \mbox{for all $t\geq t_0$},\label{bound0053}
\end{eqnarray}
for some positive constant $\bar{\nu}_1$. From Lemma \ref{lemma 1thm2} (relation (\ref{lemrel1})) we conclude that 
\begin{eqnarray}
\|\hat{\tilde{v}}(t)\|_{H^1}+\|\tilde{\bar{v}}(t)\|_{H^1}&\leq& \rho\left(\|\hat{\tilde{{v}}}(t_0)\|_{H^1}+\|\tilde{u}(t_0)\|_{H^1}\right)\nonumber\\
&&\times e^{ -\bar{\mu} \left(t-t_0\right)}, \quad \mbox{for all $t\geq t_0$},\label{bound00531}
\end{eqnarray}
where the class $\mathcal{K}_{\infty}$ function $\rho$ is given by $\rho(s)=\bar{\nu}_1 s+\bar{\nu}_1\alpha_1(s)$. Since $\sup_{x \in[0,1]}|\theta(x,t)|\leq 2\|\theta(t)\|_{H^1}$, for any $\theta\in H^1(0,1)$, choosing any positive constant $\mu^*$ such that $\mu^*\leq{\rho}^{-1}\left(\frac{c}{2}\right)$, for some $0<c<1$, we get that (\ref{feasibility}) holds. Thus, using Lemma \ref{lemma 2thm2} (relation (\ref{lemrel1n})) we get (\ref{tg}).

%similarly to the corresponding part within the proof of Theorem \ref{theorem 1}, 

Similarly to Theorem \ref{theorem 2}, due to the regularity properties of the control and observer kernels, the well-posedness of the closed-loop system is studied using the $\left(\bar{w},\hat{w}_1\right)$ system (\ref{z1})--(\ref{zn}), (\ref{sys0ervwhat})--(\ref{sysnervwhat}), with initial condition $\left(\bar{w}_0,\hat{w}_{1_0}\right)\in H^2\left(-\frac{1}{2},\frac{1}{2}\right)\times H^2\left(-\frac{1}{2},\frac{1}{2}\right)$, which satisfies the compatibility conditions. Well-posedness of the $\left(\bar{w},\hat{w}_1\right)$ system may be established with, e.g., \cite{brezis}, following the arguments employed in, e.g., \cite{vazquez2}, (see also \cite{andrey observer}, \cite{vazquez0}) and exploiting the cascade form of $\left(\bar{w},\hat{w}_1\right)$ together with the regularity of $\bar{w}$.\qed
\end{proof}

\section{Application to Traffic Flow Control}
\label{traffic app}
\subsection{Model description}
Consider a highway stretch with inlet at $x=0$ and outlet at $x=1$. We model the traffic density dynamics within the stretch with a conservation law PDE. In order to account for drivers' look-ahead ability, we incorporate in the expression for the traffic flow, in addition to the term that corresponds to a conventional fundamental diagram relation between speed and density of vehicles, an additional term that depends on the spatial derivative of the traffic density, giving rise to the following model, see, e.g., \cite{pachroo}, \cite{treiber}
\begin{eqnarray}
{\rho}_t(x,t)+\left(\rho(x,t)V\left(\rho(x,t)\right)-\epsilon \rho_x(x,t)\right)_x&=&0\label{sys1tr}\\
\rho(0,t)&=&-U_0(t)\label{sys2tr}\\
\rho\left(1,t\right)&=&-U_1(t),\label{sys3tr}
\end{eqnarray}
where, for Greenshield's fundamental diagram \cite{greenshields} we have
\begin{eqnarray}
V\left(\rho\right)=a\left(b-\rho\right),
\end{eqnarray}
with $a$, $b$ being free-flow speed and maximum density, respectively, whereas $\rho$ denotes the traffic density. The density at the boundaries may be imposed manipulating either the flow or the speed of vehicles, via the employment of ramp-metering (RM) and variable speed limits (VSL), as well as exploiting the capabilities of connected and automated vehicles see, e.g., \cite{carlson}, \cite{roncoli}.

In order to bring model (\ref{sys1tr})--(\ref{sys3tr}) into the form (\ref{sys1})--(\ref{sys3}) we define the following variable
\begin{eqnarray}
u(x,t)&\!=\!&\int_x^1\rho(y,t)dy+\int_0^tQ\left(\rho(1,s),\rho_x(1,s)\right)ds\label{transform tr}\\
Q\left(\rho,\rho_x\right)&\!=\!&\rho V\left(\rho\right)-\epsilon \rho_x.\label{flow 1}
\end{eqnarray}
It can be shown, by direct differentiation of (\ref{transform tr}) with respect to $t$ and $x$, and by employing (\ref{sys1tr}), that the variable $u$ satisfies (\ref{sys1})--(\ref{sys3}). The state $u$ represents the so-called Moskowitz function, which constitutes an alternative macroscopic description of the dynamics of traffic flow in a highway. In particular, the value of the Moskowitz function $M=u(x,t)$ is interpreted as the ``label" of a given vehicle at position $x$ at time $t$, along a road segment \cite{bayen1}, \cite{newell}. 

\subsection{Design and motivation of the feedforward/feedback control laws}
A typical aim of a traffic control scheme is to regulate the outlet flow to a certain set-point, say $q^*$, which may be the point that achieves the maximum flow (capacity flow) \cite{carlson}. In terms of the $u$ variable this corresponds to $u(1,t)$ tracking the reference trajectory $q^* t$. This motivates the trajectory generation and tracking problems for the class of systems described by (\ref{sys1})--(\ref{sys3}). Moreover, since the value $u_x(1,t)$ could be also assigned, one may choose for reference value of $-u_x(1,t)$ the value of the density that corresponds to the critical density (i.e., the density at which capacity flow is achieved) of the nominal fundamental diagram relation (i.e., when there is no $\rho_x$ term in (\ref{flow 1})) between flow and density at the outlet of the considered stretch, which in turn would guarantee that the obtained desired profile for $u_x$ (or, for $\rho$) is uniform with respect to space. Setting $a=b=1$, we obtain $y_1^{\rm r}(t)=\frac{1}{4}t$ and $y_2^{\rm r}(t)=-\frac{1}{2}$. Relations (\ref{ref12})--(\ref{y2v}) are then written for $x_0=1$ as
\begin{eqnarray}
%v^{\rm r}(x,t)&=&\sum_{k=0}^{\infty}\left(\mathcal{D}^{k}\left\{y_{1,v}^{\rm r}(t)\right\}\frac{\left(x-1\right)^{2k}}{2k!}+\mathcal{D}^{k}\left\{y_{2,v}^{\rm r}(t)\right\}\frac{\left(x-1\right)^{2k+1}}{\left(2k+1\right)!}\right)\label{ref12}\\
v^{\rm r}(x,t)&=&\sum_{k=0}^{\infty}\frac{1}{\epsilon^k} \frac{\left(x-1\right)^{2k}}{2k!}\sum_{m=0}^{k}\binom {k} {m}\left(\frac{1}{4\epsilon}\right)^{k-m}\nonumber\\
&&\times{y_{1,v}^{\rm r}}^{(m)}(t)+\sum_{k=0}^{\infty}\frac{1}{\epsilon^{k}} \frac{\left(x-1\right)^{2k+1}}{\left(2k+1\right)!}\nonumber\\
&&\times\sum_{m=0}^{k}\binom {k} {m}\left(\frac{1}{4\epsilon}\right)^{k-m}{y_{2,v}^{\rm r}}^{(m)}(t)\label{ref1211}\\
y_{1,v}^{\rm r}(t)&=&e^{-\frac{1}{2\epsilon}}\left(e^{-\frac{1}{4\epsilon}t}-1\right)\label{y1v11}\\
y_{2,v}^{\rm r}(t)&=&\frac{1}{2\epsilon}e^{-\frac{1}{2\epsilon}},\label{y2v11}
\end{eqnarray}
and hence,
\begin{eqnarray}
%v^{\rm r}(x,t)&=&\sum_{k=0}^{\infty}\left(\mathcal{D}^{k}\left\{y_{1,v}^{\rm r}(t)\right\}\frac{\left(x-1\right)^{2k}}{2k!}+\mathcal{D}^{k}\left\{y_{2,v}^{\rm r}(t)\right\}\frac{\left(x-1\right)^{2k+1}}{\left(2k+1\right)!}\right)\label{ref12}\\
v^{\rm r}(x,t)&=&e^{-\frac{1}{2\epsilon}}\left(e^{-\frac{1}{4\epsilon}t}-e^{\frac{1-x}{2\epsilon}}\right)\label{y1v112}.
\end{eqnarray}
Therefore, employing (\ref{166})--(\ref{til2tra}) the reference trajectory and reference inputs are given explicitly as
\begin{eqnarray}
%v^{\rm r}(x,t)&=&\sum_{k=0}^{\infty}\left(\mathcal{D}^{k}\left\{y_{1,v}^{\rm r}(t)\right\}\frac{\left(x-1\right)^{2k}}{2k!}+\mathcal{D}^{k}\left\{y_{2,v}^{\rm r}(t)\right\}\frac{\left(x-1\right)^{2k+1}}{\left(2k+1\right)!}\right)\label{ref12}\\
u^{\rm r}(x,t)&=&\frac{1}{4}t+\frac{1-x}{2}\label{y1v112n}\\
U_0^{\rm r}(t)&=&U_1^{\rm r}(t)\!=\!-\frac{1}{2}.
\end{eqnarray}
The feedback control laws are given in (\ref{con1}), (\ref{con2}) with $c_1\!=\!1$.

Note that although the reference trajectory (\ref{y1v112n}) doesn't satisfy the conditions of Theorem \ref{theorem 1} (since $y_1^{\rm r}(t)=\frac{1}{4}t$ is not uniformly bounded) trajectory tracking is achieved, which is explained as follows. The trajectory tracking problem is solvable provided that stabilization of the zero equilibrium of system (\ref{error1})--(\ref{error2}) implies stabilization of system (\ref{sys0erver})--(\ref{sysnerver}), which is possible when relations (\ref{lemrel1nne}) and (\ref{lemrel1nne1}) hold. In the case of the reference trajectory given by (\ref{y1v112n}) relation (\ref{lemrel1nne}) holds, but relation (\ref{lemrel1nne1}) does not. However, since from relation (\ref{transformation1add}) it holds that $\tilde{\bar{v}}(x,t)=\tilde{{v}}(x,t)e^{\frac{1}{2\epsilon}+\frac{1}{4\epsilon}t}$, stabilization of system (\ref{sys0erver})--(\ref{sysnerver}) is achieved provided that the convergence rate of the $H_1$ norm of $\tilde{v}$ is larger than $\frac{1}{4\epsilon}$, which holds true whenever $c_1>0$ (that would also imply from (\ref{bac1}), (\ref{bac2}) that the control inputs (\ref{convtil1}), (\ref{sysnerver}) are bounded). This in turn implies that in order for stabilization to be achieved the full-state feedback control laws should be employed, whereas when $c_1=0$ the closed-loop system is not asymptotically stable. To see this, note that because $u_x^{\rm r}(x,t)=-\frac{b}{2}$ system (\ref{sys0erver})--(\ref{sysnerver}) reduces to $\tilde{\bar{v}}_t(x,t)=\epsilon \tilde{\bar{v}}_{xx}(x,t)$, $\tilde{\bar{v}}_x(0,t)=\tilde{\bar{v}}_x(1,t)=0$ when $\tilde{V}_0(t)=\tilde{V}_1(t)=0$. This strengthens the motivation for the design of the bilateral, full-state feedback controllers. 

\subsection{Trajectory tracking}
We choose $\epsilon=0.25$, whereas the initial condition is defined as $u(x,0)=u^{\rm r}(x,0)+0.1\textrm{sin}\left(\pi x\right)=\frac{1-x}{2}+0.1\textrm{sin}\left(\pi x\right)$. In Fig. \ref{fig1} we show the output $u(1,t)$, from which it is evident that asymptotic trajectory tracking is achieved.\begin{figure}[h]
\centering
\includegraphics[width=\linewidth]{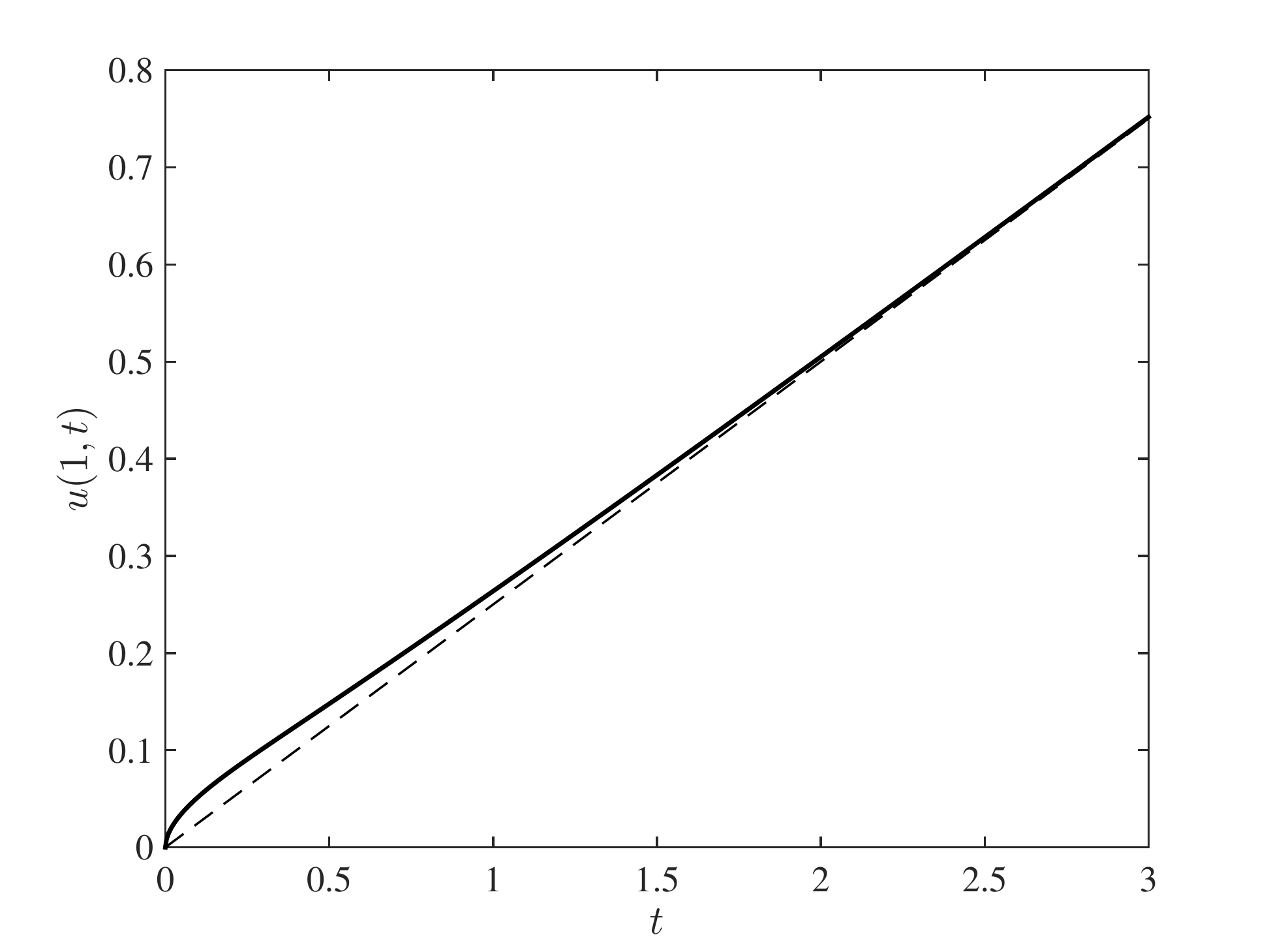}
\caption{Solid line: The output $u(1,t)$ of system (\ref{sys1})--(\ref{sys3}) with $a=b=1$, $\epsilon=0.25$, under the feedback laws (\ref{con1}), (\ref{con2}), (\ref{166})--(\ref{til2tra}) with $c_1=1$ for initial condition $u(x,0)=\frac{1-x}{2}+0.1\textrm{sin}\left(\pi x\right)$. Dashed line: The reference output $u^{\rm r}(1,t)=\frac{1}{4}t$.}
\label{fig1}
\end{figure}
In Fig. \ref{fig4}, we show the highway density $\rho(x,t)$. One can observe that the density converges to the desired reference profile, namely, to the uniform profile $\rho^{\rm e}(x)=\frac{1}{2}$, for all $x\in[0,1]$. Note that the output $u_x(1,t)$ equals $-\rho(1,t)$ and, according to Fig. \ref{fig4}, converges to $u_x^{\rm r}(1,t)=-\frac{1}{2}$. \begin{figure}[h]
\centering
\includegraphics[width=\linewidth]{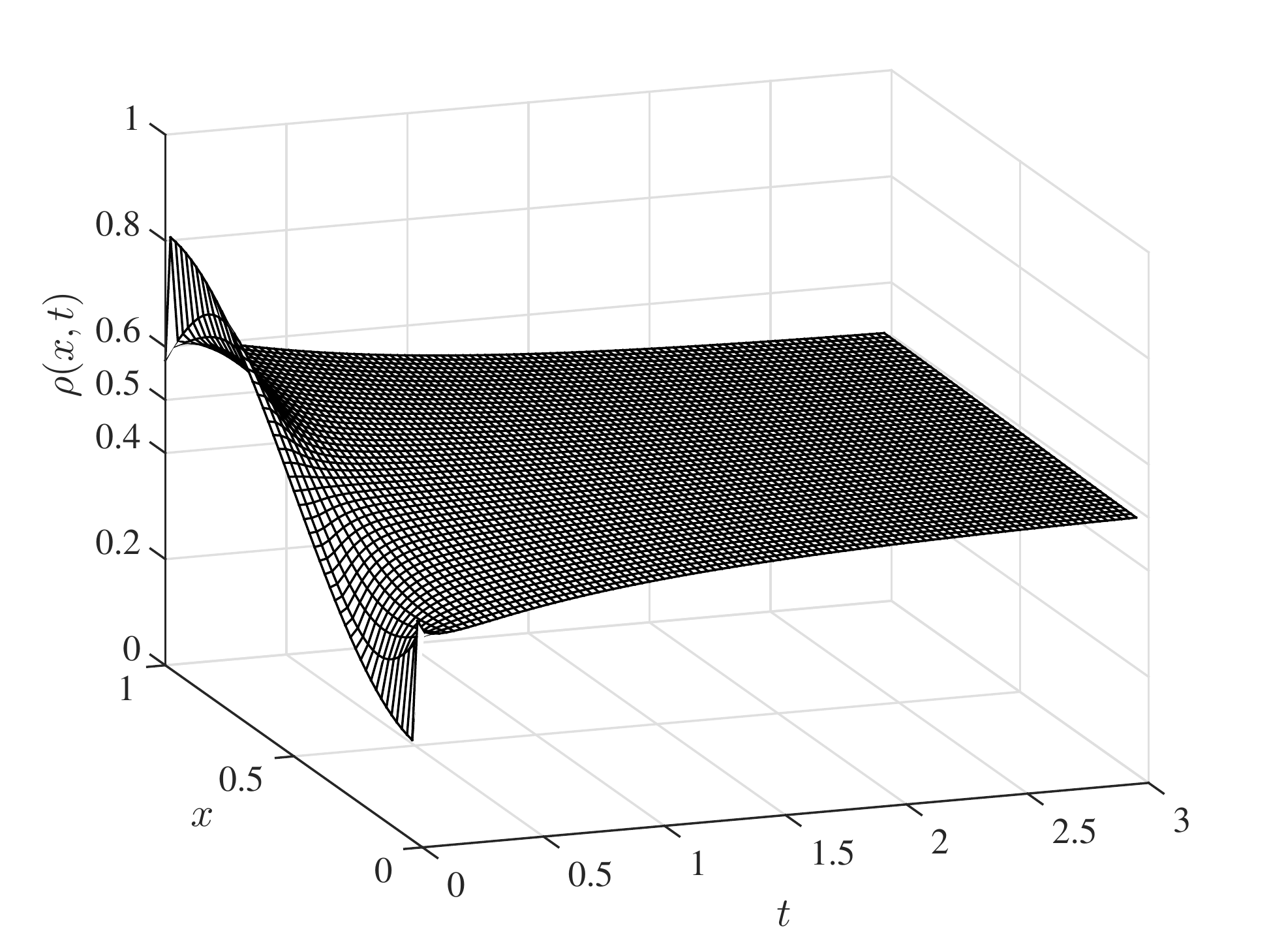}
\caption{The density evolution of the highway stretch.}
\label{fig4}
\end{figure}

\subsection{Control effort comparison with the unilateral case}
In Fig. \ref{fig2} we show the control efforts (\ref{con1}), (\ref{con2}), (\ref{166})--(\ref{til2tra}) of the bilateral boundary control design as well as the control efforts in the unilateral case, in which, a full-state feedback law is employed only at the one boundary, while, at the other end, only the static, collocated output feedback law (\ref{static1}) is applied (for the same initial conditions and reference outputs). The control laws in the unilateral case are designed such that the same decay rate for the closed-loop system is obtained (or, in other words, the same target system $w$ is obtained). The unilateral backstepping controller is derived from (\ref{con2}) replacing $k$ by the kernel $k_1(x,y)=-\frac{c_1}{\epsilon}x\frac{I_1\left(\sqrt{\frac{c_1}{\epsilon}\left(x^2-y^2\right)}\right)}{\sqrt{\frac{c_1}{\epsilon}\left(x^2-y^2\right)}}$ (see, e.g., \cite{smyshlyaev}), whereas in the present numerical example, the control law (\ref{static1}) simplifies to the reference input\footnote{Similarly, one could apply a backstepping controller only at the end $x=0$.}. The unilateral control laws are then given as
\begin{eqnarray}
{U}^{\rm uni}_0(t)&=&U_0^{\rm r}(t)\label{first uni}\\
{U}^{\rm uni}_1(t)&=&-\epsilon e^{\frac{1}{\epsilon}{\tilde{u}}(1,t)}\left(\vphantom{\int_{0}^1K_x\left(\frac{1}{2},\xi-\frac{1}{2}\right)}\left(k_1\left(1,1\right)+\frac{1}{2\epsilon}\right)\left(e^{-\frac{1}{\epsilon}\tilde{u}(1,t)}-1\right)\right.\nonumber\\
&&\left.\vphantom{\int_{0}^1K_x\left(\frac{1}{2},\xi-\frac{1}{2}\right)}+e^{\frac{1}{2\epsilon}+\frac{1}{\epsilon}u^{\rm r}\left(1,t\right)}\int_{0}^1k_{1_x}\left(1,\xi\right)e^{-\frac{1}{2\epsilon}\xi-\frac{1}{\epsilon}u^{\rm r}\left(\xi,t\right)}\right.\nonumber\\
&&\left.\times\left(e^{-\frac{1}{\epsilon}\tilde{u}\left(\xi,t\right)}-1\right)d\xi\vphantom{\int_{0}^1K_x\left(\frac{1}{2},\xi-\frac{1}{2}\right)}\right)+U_1^{\rm r}(t)e^{\frac{1}{\epsilon}{\tilde{u}}(1,t)}.\label{con2ex1}
\end{eqnarray}
From Fig. \ref{fig2} it is evident that the unilateral control design results in larger control effort, although the convergence rate of the closed-loop system would be identical to the bilateral case. Thus, although in both cases actuation is applied at both ends, the bilateral control design results in a feedback law that utilizes more efficiently both the available actuators and the available measurements. It should be also noted that, from a traffic flow control perspective, such large control values may lead to practically unrealistic ordered values for flows or speeds.

\begin{figure}[h]
\centering
\includegraphics[width=\linewidth]{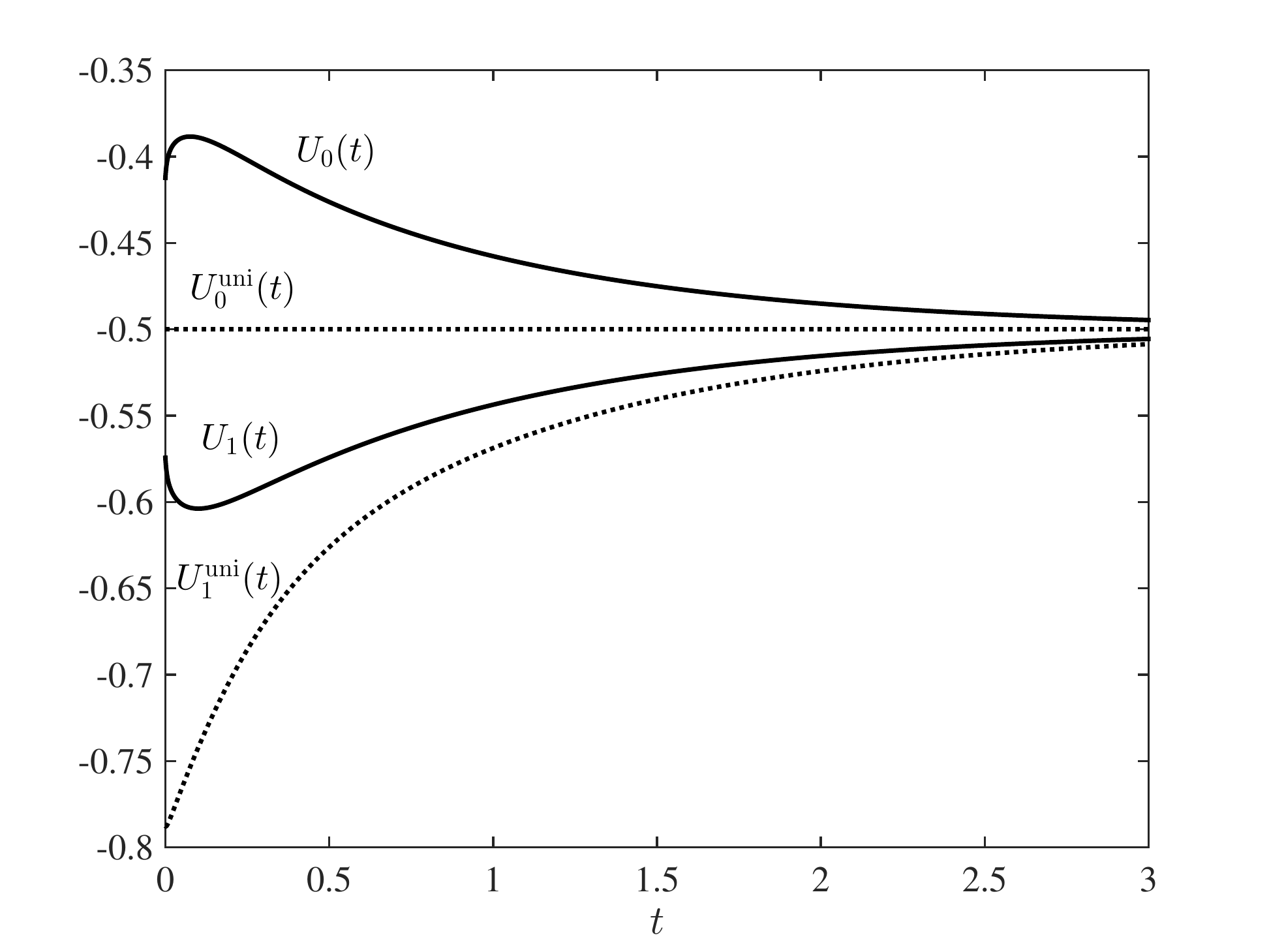}
\caption{Solid lines: Control efforts (\ref{con1}), (\ref{con2}), (\ref{166})--(\ref{til2tra}). Dotted lines: Control efforts (\ref{first uni}), (\ref{con2ex1}) of the unilateral controllers.}
\label{fig2}
\end{figure}

%Moreover, in Fig. \ref{fige}, we show the outlet flow, namely the function $u_t(1,t)$, which converges to the desired value $u_t(1,t)=\frac{1}{4}$.
%\begin{figure}[h]
%\centering
%\includegraphics[width=.5\linewidth]{flowtra}
%\includegraphics[width=.5\linewidth]{U11tr}
%\caption{The flow evolution at the outlet of the highway stretch.}
%\label{fige}
%\end{figure}
%In Fig. \ref{fig tra 1} we 

% which implies in terms of the density $\rho$ that $\rho(x,0)=-u_x^{\rm r}(x,0)+0.4\textrm{cos}\left(4x\right)=0.5+0.4\textrm{cos}\left(4x\right)$. The initial condition $rho(x,0)$ is shown in Fig. \ref{fig init}. Note that this initial condition for the density implies that, initially, part of the highway is congested. \begin{figure}[h]
%\centering
%\includegraphics[width=.5\linewidth]{trinit}
%\caption{Initial density profile in the considered highway stretch.}
%\label{fig init}
%\end{figure}

\section{Conclusions}
\label{conclude}
For a class of viscous HJ PDEs with actuation and sensing at both boundaries we, 1) solved the nonlinear trajectory generation problem, 2) presented nonlinear, bilateral full-state feedback control designs, 3) constructed a nonlinear observer as well as observer-based output-feedback controllers, 4) established local asymptotic stability of the closed-loop systems under the developed controllers, 5) illustrated our results in simulation via a traffic flow control example.

As a potential topic of future research one may consider problems that involve interconnections of viscous HJ PDEs with Ordinary Differential Equations (ODEs), as it is the case, for example, in \cite{hasan1}, which considers an interconnected system consisting of a viscous Burgers PDE and a linear ODE. The bilateral backstepping design used in this work can potentially deal with more complex PDE-ODE couplings than the standard unilateral design, thus we expect to be able to consider new families of previously unexplored systems. Another possible next step may be problems that incorporate viscous HJ PDE systems with actuator (or sensor) dynamics governed by certain types of ODEs or PDEs, as it is the case with, e.g., \cite{meurer fred}, \cite{liu}, which are dealing with viscous Burgers PDEs with ODE input dynamics.

\setcounter{equation}{0}
\renewcommand{\theequation}{A.\arabic{equation}}
\setcounter{lemma}{0}
\renewcommand{\thelemma}{A.\arabic{lemma}}
%\appendices
\section*{Appendix A}

\subsection*{Technical lemmas}

\begin{lemma}
\label{lemma1}
Let $f(t)$ be in $G_{F,M,\gamma}\left([0,+\infty)\right)$ with $\gamma\in[1,2)$. Then the function $g(t)=e^{f(t)}-1$ belongs to $G_{F_1,M_1,\gamma}\left([0,+\infty)\right)$ with $F_1=Fe^F$ and $M_1=Me^{F}$. %$M_1=Me^{M}\max\left\{M,2\right\}$.
\end{lemma}

\begin{proof}
From the power series expansion of the exponential function and the triangular inequality we obtain that
\begin{eqnarray}
\left|g^{(n)}(t)\right|&\leq&\sum_{k=1}^{\infty}\frac{\left|{\frac{d^nf(t)^k}{dt^n}}\right|}{k!}.\label{y1vnewap}
\end{eqnarray}
We claim that for any $k$-th power of $f$ the following holds 
\begin{eqnarray}
\sup_{t\geq0}\left|\frac{d^nf(t)^k}{dt^n}\right|&\leq&(n+1)^{k-1}F^{k}M^n\left(n!\right)^{\gamma},\nonumber\\
&& \mbox{$n=0,1,\ldots$},\label{y1vnewap1}
\end{eqnarray}
which we prove by induction. For $k=1$ our claim is true by assumption. Assume next that relation (\ref{y1vnewap1}) holds for $k>1$. We show below that it holds for $k+1$. Toward that end, we employ Leibniz formula for the $n$-th derivative of the product of two functions to get that
\begin{eqnarray}
\left|\frac{d^nf(t)^{k+1}}{dt^n}\right|&=&\left|\frac{d^n \left( f(t)^{k}f(t)  \right)}{dt^n}\right|\nonumber\\
&=&\left|\sum_{i=0}^{n}\binom{n}{i}f^{(i)}(t)\frac{d^{n-i} \left( f(t)^{k}  \right)}{dt^{n-i}}\right|,\label{y1vnewap2}
\end{eqnarray}
and hence, using (\ref{y1vnewap1}) we obtain
\begin{eqnarray}
\left|\frac{d^nf(t)^{k+1}}{dt^n}\right|&\leq&\sum_{i=0}^{n}\binom{n}{i}\left|f^{(i)}(t)\right|(n-i+1)^{k-1}F^{k}\nonumber\\
&&\times M^{n-i}\left((n-i)!\right)^{\gamma},\label{y1vnewap3}
\end{eqnarray}
which, under the assumption that $f\in G_{F,M,\gamma}\left([0,+\infty)\right)$, in turn implies that
\begin{eqnarray}
\left|\frac{d^nf(t)^{k+1}}{dt^n}\right|&\leq&F^{k+1}M^{n}(n+1)^{k-1}\nonumber\\
&&\times\sum_{i=0}^{n}\binom{n}{i}\left(i!\right)^{\gamma}\left((n-i)!\right)^{\gamma}.\label{y1vnewap4}
\end{eqnarray}
From the definition of the binomial coefficient we obtain from (\ref{y1vnewap4}) that
\begin{eqnarray}
\left|\frac{d^nf(t)^{k+1}}{dt^n}\right|&\leq&(n+1)^{k-1}F^{k+1}M^{n}\left(n!\right)^{\gamma}\nonumber\\
&&\times\sum_{i=0}^{n}\frac{\left(i!\right)^{\gamma-1}\left((n-i)!\right)^{\gamma-1}}{\left(n!\right)^{\gamma-1}}\nonumber\\
&=&(n+1)^{k-1}F^{k+1}M^{n}\left(n!\right)^{\gamma}\nonumber\\
&&\times\sum_{i=0}^{n}\binom{n}{i}^{1-\gamma},\label{y1vnewap5}
\end{eqnarray}
which gives 
\begin{eqnarray}
\left|\frac{d^nf(t)^{k+1}}{dt^n}\right|&\leq&(n+1)^kF^{k+1}(M)^{n}\left(n!\right)^{\gamma},\label{y1vnewap6}
\end{eqnarray}
where we used the fact that $\sum_{i=0}^{n}\binom{n}{i}^{1-\gamma}\leq\sum_{i=0}^{n}1=n+1$, for $\gamma\in[1,2)$. Therefore, relation (\ref{y1vnewap1}) is established, which implies from (\ref{y1vnewap}) that
\begin{eqnarray}
\left|g^{(n)}(t)\right|&\leq&M^{n}\left(n!\right)^{\gamma}\frac{1}{n+1}\sum_{k=1}^{\infty}\frac{F^{k}\left(n+1\right)^k}{k!},\label{y1vnewap7}
\end{eqnarray}
and hence,
\begin{eqnarray}
\left|g^{(n)}(t)\right|&\leq&M^{n}\left(n!\right)^{\gamma}\frac{1}{n+1}\left(e^{F(n+1)}-1\right).\label{y1vnewap89}
\end{eqnarray}
Using the fact that $e^r-1\leq re^r$, for all $r\geq0$ we get the following estimate for all $n=0,1,2,\ldots$
%Using the facts that $e^r-1\leq re^r$, for all $r\geq0$ and that $n+1\leq 2^n$, for all $n=0,1,\ldots$ we get the following estimate
\begin{eqnarray}
\sup_{t\geq0}\left|{g}^{(n)}(t)\right|&\leq& Fe^F \left(e^FM\right)^{n}\left(n!\right)^{\gamma},\label{nik}
\end{eqnarray}
which concludes the proof.
\qed
\end{proof}

\begin{lemma}
\label{lemma2}
Let $\bar{f}(t)$ and $\bar{g}(t)$ be in $G_{F,M,\gamma}\left([0,+\infty)\right)$ with $\gamma\in[1,2)$. Then the function $h(t)=e^{\bar{f}(t)}\bar{g}(t)$ belongs to $G_{F_2,M_2,\gamma}\left([0,+\infty)\right)$ with $F_2=F\left(1+Fe^F\right)$ and $M_2=\left(1+Fe^F\right)Me^F$.
\end{lemma}

\begin{proof}
We start by writing the function $h$ as $h(t)=\left(e^{\bar{f}(t)}-1\right)\bar{g}(t)+\bar{g}(t)$. Using Leibniz formula for the $n$-th derivative of the product of two functions we get $\left|{h}^{(n)}(t)\right|=\left|\sum_{i=0}^{n}\binom{n}{i}\bar{g}^{(i)}(t)\frac{d^{n-i} \left( e^{\bar{f}(t)}-1 \right)}{dt^{n-i}}\vphantom{\sum_{i=0}^{n}\binom{n}{i}\bar{g}^{(i)}(t)\frac{d^{n-i} \left( e^{\bar{f}(t)}-1 \right)}{dt^{n-i}}}+\bar{g}^{(n)}(t)\right|$, and hence, from Lemma \ref{lemma1} we obtain
\begin{eqnarray}
\left|{h}^{(n)}(t)\right|&\leq&\sum_{i=0}^{n}\binom{n}{i}FM^{i}\left(i!\right)^{\gamma}F_1M_1^{n-i}\left((n-i)!\right)^{\gamma}\nonumber\\
&&+FM^{n}\left(n!\right)^{\gamma}\nonumber\\
&\leq&FF_1{M^*}^{n}\sum_{i=0}^{n}\binom{n}{i}\left(i!\right)^{\gamma}\left((n-i)!\right)^{\gamma}\nonumber\\
&&+F{M^*}^{n}\left(n!\right)^{\gamma},\label{y1vnewap9559}
\end{eqnarray}
where $M^*=\max\left\{M,M_1\right\}$. With similar arguments to the proof of Lemma \ref{lemma1} and employing Bernoulli's inequality we get the following estimate for all $n=0,1,2,\ldots$
\begin{eqnarray}
\sup_{t\geq0}\left|{h}^{(n)}(t)\right|&\leq& F\left(1+F_1\right){M^*}^n\left(1+F_1\right)^{n}\left(n!\right)^{\gamma}.\label{nikpr}
\end{eqnarray}
The proof is completed using the expressions for $M_1$ and $F_1$ from Lemma \ref{lemma1}. 
\qed
\end{proof}

\setcounter{equation}{0}
\renewcommand{\theequation}{B.\arabic{equation}}
\setcounter{lemma}{0}
\renewcommand{\thelemma}{B.\arabic{lemma}}
%\appendices
\section*{Appendix B}

\subsection*{Proof of Lemma \ref{lemma 1thm2}}

Using the fact that for the function $f(r)=e^{-\frac{a}{\epsilon}r}-1$ it holds $|f(r)|\leq\frac{|a||r|}{\epsilon}e^{\frac{|a||r|}{\epsilon}}$, $\forall r\in\mathbb{R}$, we get from (\ref{transformation1er}) that $|\tilde{\bar{v}}(x,t)|\leq\hat{\alpha}\left(|\tilde{u}(x,t)|\right)$, where $\hat{\alpha}(s)=\frac{|a|s}{\epsilon}e^{\frac{|a|s}{\epsilon}}\in\mathcal{K}_{\infty}$. Hence, $|\tilde{\bar{v}}(x,t)|\leq\hat{\alpha}\left(\sup_{x\in[0,1]}|\tilde{u}(x,t)|\right)$, $\forall x\in[0,1]$. For any $u\in H^1(0,1)$ we have $u(x,t)=u(0,t)+\int_0^xu_y(y,t)dy$. Hence, using Cauchy-Schwartz's inequality we get
\begin{eqnarray}
|u(x,t)|\leq |u(0,t)|+\sqrt{\int_0^1u_x(x,t)^2dx},\quad x\in[0,1].\label{rel03}
\end{eqnarray}
Since $u(0,t)=u(x,t)-\int_0^xu_y(y,t)dy$ we get $|u(0,t)|\leq|u(x,t)|+\sqrt{\int_0^1u_x(x,t)^2dx}$. Hence, by integrating we get 
\begin{eqnarray}
|u(0,t)|\leq\|u(t)\|_{H^1}. \label{rel23}
\end{eqnarray}
Since $\forall x$, $|\tilde{\bar{v}}(x,t)|\leq\hat{\alpha}\!\left(\sup_{x\in[0,1]}|\tilde{u}(x,t)|\right)$ and by  (\ref{rel03}), (\ref{rel23})
\begin{eqnarray}
\sup_{x\in[0,1]}|\tilde{u}(x,t)|\leq 2\|\tilde{u}(t)\|_{H^1},\label{sh}
\end{eqnarray}
we get $\sup_{x\in[0,1]}|\tilde{\bar{v}}(x,t)|\leq\hat{\alpha}\left(2\|\tilde{u}(t)\|_{H^1}\right)$. Thus,
\begin{eqnarray}
\|\tilde{\bar{v}}(t)\|_{L^2}\leq\hat{\alpha}\left(2\|\tilde{u}(t)\|_{H^1}\right).\label{relation0}
\end{eqnarray}
Differentiating (\ref{transformation1er}) we get with (\ref{sh}) that $\sqrt{\int_0^1\tilde{\bar{v}}_x(x,t)^2dx}$ $\leq\frac{|a|}{\epsilon}e^{\frac{2|a|}{\epsilon}\|\tilde{u}(t)\|_{H^1}}\sqrt{\int_0^1\tilde{u}_x(x,t)^2dx}$, and hence, with (\ref{relation0}) we get 
\begin{eqnarray}
\|\tilde{\bar{v}}(t)\|_{H^1}&\leq& \hat{\alpha}\left(2\|\tilde{u}(t)\|_{H^1}\right)\nonumber\\
&&+\frac{|a|}{\epsilon}e^{\frac{2|a|}{\epsilon}\|\tilde{u}(t)\|_{H^1}}\|\tilde{u}(t)\|_{H^1}.\label{boundnew}
\end{eqnarray}
The proof is completed with $\alpha_1(s)= \hat{\alpha}\left(2s\right)+\frac{|a|}{\epsilon}e^{\frac{2|a|}{\epsilon}s}s$.

\subsection*{Proof of Lemma \ref{lemma 2thm2}}

Using (\ref{transformationinv}), (\ref{feasibility}) we get
\begin{eqnarray}
|\tilde{u}(x,t)|\leq \frac{\epsilon}{|a|(1-c)}|\tilde{\bar{v}}(x,t)|, 
\end{eqnarray}
since $|\textrm{ln}\left(r+1\right)|\leq \frac{1}{1-c}|r|$, $\forall|r|< c$, $0<c<1$. Hence, $\|\tilde{u}(t)\|_{L^2}\leq \frac{\epsilon}{|a|(1-c)}\|\tilde{\bar{v}}(t)\|_{L^2}$. Using this relation and (\ref{feasibility}) we arrive at (\ref{lemrel1n}) by differentiating (\ref{transformationinv}) with respect to $x$.
%\begin{eqnarray}
%\|\tilde{u}(t)\|_{H^1}\leq \alpha_3\left(\|\tilde{v}(t)\|_{H^1}\right).\label{lemrel1n}
%\|\tilde{u}(t)\|_{H^1}\leq \frac{\epsilon}{1-c}\|\tilde{v}(t)\|_{H^1}.\label{lemrel1n}
%\end{eqnarray}
%The proof is completed combining estimate (\ref{bound00}) with estimates (\ref{boundnew}), (\ref{lemrel1n}).

\subsection*{Proof of Lemma \ref{lemma 3thm2}}

From (\ref{transformation1add}) and (\ref{166}) it follows that 
\begin{eqnarray}
\tilde{v}(x,t)=\tilde{\bar{v}}(x,t)e^{-\frac{ab}{2\epsilon}x}\left(e^{\frac{ab}{2\epsilon}x}v^{\rm r}(x,t)+1\right). \label{rel1}
\end{eqnarray}
Under the conditions of Theorem \ref{theorem 1}, which also guarantee that $v^{\rm r}$ is uniformly bounded with respect to time and spatial variable (see relation (\ref{Res1})), we obtain for all $x\in[0,1]$ and $t\geq t_0$
\begin{eqnarray}
\tilde{v}(x,t)^2\leq\nu_1\tilde{\bar{v}}(x,t)^2,
\end{eqnarray}
for some positive constant $\nu_1$, and hence,
\begin{eqnarray}
\|\tilde{v}(t)\|_{L^2}^2\leq\nu_1\|\tilde{\bar{v}}(t)\|_{L^2}^2.\label{bound2}
\end{eqnarray}
Moreover, differentiating (\ref{rel1}) with respect to $x$ we get that
\begin{eqnarray}
\tilde{v}_x(x,t)&=&\tilde{\bar{v}}_x(x,t)\left(v^{\rm r}(x,t)+e^{-\frac{ab}{2\epsilon}x}\right)\nonumber\\
&&+\tilde{\bar{v}}(x,t)\left(v_x^{\rm r}(x,t)-\frac{ab}{2\epsilon}e^{-\frac{ab}{2\epsilon}x}\right). \label{rel2}
\end{eqnarray}
Mimicking the arguments of boundedness for $v^{\rm r}$ in the proof of Theorem \ref{theorem 1}, it is shown that $v_x^{\rm r}$ is uniformly bounded with respect to time and spatial variable, and thus, it follows from (\ref{rel2}) that
\begin{eqnarray}
\|\tilde{v}_x(t)\|_{L^2}^2\leq\nu_2\|\tilde{\bar{v}}(t)\|_{H^1}^2, \label{rel3}
\end{eqnarray}
for some positive constant $\nu_2$. Combining (\ref{bound2}) with (\ref{rel3}) we arrive at (\ref{lemrel1nne}) with $\xi_1=\sqrt{\nu_1}+\sqrt{\nu_2}$. 

Under the conditions of Theorem \ref{theorem 1}, using relations (\ref{Res1}), (\ref{rel1}) we obtain from (\ref{rel1}) for all $x\in[0,1]$ and $t\geq t_0$
\begin{eqnarray}
\|\tilde{\bar{v}}(t)\|_{L^2}^2\leq\frac{e^{\left|\frac{ab}{\epsilon}\right|}}{\left(1-\bar{c}\right)^2}\|\tilde{{v}}(t)\|_{L^2}^2.\label{bound1}
\end{eqnarray}
With (\ref{Res1}), solving (\ref{rel2}) with respect to $\tilde{\bar{v}}_x$ and employing (\ref{bound1}) we show that
\begin{eqnarray}
\|\tilde{\bar{v}}_x(t)\|_{L^2}^2\leq\nu_3\|\tilde{{v}}(t)\|_{H^1}^2, \label{rel4}
\end{eqnarray}
for some positive constant $\nu_3$. Combining (\ref{bound1}) with (\ref{rel4}) we get estimate (\ref{lemrel1nne1}) with $\xi_2=\sqrt{\nu_3}+\frac{e^{\left|\frac{ab}{2\epsilon}\right|}}{1-\bar{c}}$.

\setcounter{equation}{0}
\renewcommand{\theequation}{C.\arabic{equation}}
\setcounter{lemma}{0}
\renewcommand{\thelemma}{C.\arabic{lemma}}
%\appendices
\section*{Appendix C}

\subsection*{Connection of transformation (\ref{trans1}) with the one from \cite{Vazquez1}}

It is convenient to define the shifted variables $z=x-\frac{1}{2}$ and $y=\xi-\frac{1}{2}$, to re-write transformation (\ref{trans1}) as
\begin{eqnarray}
w_1(z,t)=\tilde{v}_1(z,t)-\int_{-z}^zK\left(z,y\right)\tilde{v}_1\left(y,t\right)dy,\label{trans1z}
\end{eqnarray}
where we define $w_1(z,t)=w\left(z+\frac{1}{2},t\right)$, $\tilde{v}_1(z,t)=\tilde{v}\left(z+\frac{1}{2},t\right)$, and 
\begin{eqnarray}
K(z,y)=k\left(z+\frac{1}{2},y+\frac{1}{2}\right), \label{K1}
\end{eqnarray}
with $K$ being defined for all $-z\leq y\leq z$ when $0\leq z\leq \frac{1}{2}$ and for all $z\leq y\leq -z$ when $-\frac{1}{2}\leq z\leq 0$, where $k$ is given in (\ref{poo}). Noting that the $\tilde{v}_1$ variable satisfies the same exact PDE system (\ref{error1})--(\ref{error2}) as the $\tilde{v}$ variable (with the only difference that the boundary conditions for $\tilde{v}_1$ are taken for $z=-\frac{1}{2}$ and ${z}=\frac{1}{2}$, which correspond to $x=0$ and $x=1$, respectively) and that the kernel $K$ in (\ref{K1}) is the gain kernel introduced in \cite{Vazquez1}, we conclude (see \cite{Vazquez1}) that the $w_1$ variable in (\ref{trans1z}) satisfies the PDE $w_{1_t}(z,t)=\epsilon {w}_{1_{zz}}(z,t)-\left(\frac{a^2b^2}{4\epsilon}+c_1\right)w_1(z,t)$. The boundary condition $w_{1_z}\left(-\frac{1}{2},t\right)=0$ is obtained differentiating (\ref{trans1z}) with respect to $z$ and using (\ref{bac1}). The boundary condition $w_{1_z}\left(\frac{1}{2},t\right)=0$ is obtained analogously.

\setcounter{equation}{0}
\renewcommand{\theequation}{D.\arabic{equation}}
\setcounter{lemma}{0}
\renewcommand{\thelemma}{D.\arabic{lemma}}
%\appendices
\section*{Appendix D}

\subsection*{Derivation of (\ref{esys1z})--(\ref{esys2z}) from (\ref{z1})--(\ref{zn}) via (\ref{r1}), (\ref{r2})}
Using relation (\ref{gainR}), the kernel $p$ in (\ref{pp}) can be written as
\begin{eqnarray}
%\Gamma\left(z,y\right)=-\frac{1}{2}\sqrt{\frac{c_2}{\epsilon}}{\rm I}_{1}\left(\sqrt{\frac{c_2}{\epsilon}\left(y^2-z^2\right)}\right)\sqrt{\frac{y+z}{y-z}},\label{kl}\\
p\left(z,y\right)=-\frac{1}{2}\sqrt{\frac{c_2}{\epsilon}}\frac{{\rm I}_{1}\left(\sqrt{\frac{c_2}{\epsilon}\left(y^2-z^2\right)}\right)}{\sqrt{y^2-z^2}}\left(y+z\right),\label{kl}
\end{eqnarray}
which is defined in the domain $\bar{E}=\bar{E}_1\cup \bar{E}_2$, where $\bar{E}_1=\left\{\left(z,y\right):0\leq y\leq \frac{1}{2}, -y\leq z\leq y\right\}$ and $\bar{E}_2=\left\{\left(z,y\right):-\frac{1}{2}\leq y\leq 0, y\leq z\leq -y\right\}$. Function $p(z,y)$ satisfies the following problem in $\bar{E}_1\cup\bar{E}_2$
\begin{eqnarray}
\epsilon p_{yy}\left(z,y\right)-\epsilon p_{zz}\left(z,y\right)&=&c_2 p(z,y)\label{klg}\\
p(z,z)&=&-\frac{c_2}{2\epsilon}z\\
p(z,-z)&=&0,\label{3rel}
\end{eqnarray}
which can be either seen by direct computation (using the properties of Bessel functions, see, e.g., \cite{smyshlyaev}) or by using relations (75)--(77) from \cite{Vazquez1} and noting that $p(z,y)$ satisfies relation (78) from \cite{Vazquez1} with $x\to y$ and $\xi\to z$. Using (\ref{r1}), we next provide few details on the derivation of (\ref{esys1z}), (\ref{esys2z}) from (\ref{z1}), (\ref{zn}) (one can then derive (\ref{esys1z}), (\ref{e sys lot}) from (\ref{z1}), (\ref{zn lot}) via (\ref{r2}), employing identical arguments). In fact, the computations follow the standard, but tedious, algebraic manipulations involved in the backstepping methodology when deriving a target system for a specific case. In our case, differentiating (\ref{r1}) with respect to $t$ and $x$, and using equations (\ref{z1})--(\ref{zn}) as well as equations (\ref{klg})--(\ref{3rel}) we arrive at
\begin{eqnarray}
&&\bar{e}_t(z,t)-\epsilon \bar{e}_{zz}(z,t)+\frac{a^2b^2}{4\epsilon}\bar{e}(z,t)\nonumber\\
&&+\bar{p}_2(z)\bar{e}\left(-\frac{1}{2},t\right)+\bar{p}_1(z)\bar{e}\left(\frac{1}{2},t\right)=\nonumber\\
&&\bar{p}_2(z)\bar{e}\left(-\frac{1}{2},t\right)+\bar{p}_1(z)\bar{e}\left(\frac{1}{2},t\right)\nonumber\\
&&+\epsilon p_y\left(z,\frac{1}{2}\right)\bar{w}\left(\frac{1}{2},t\right)+\epsilon p_y\left(z,-\frac{1}{2}\right)\bar{w}\left(-\frac{1}{2},t\right)\!.\label{alm}
\end{eqnarray}
Therefore, using (\ref{r1}) and (\ref{r2}) for $z=\frac{1}{2}$ and $z=-\frac{1}{2}$, respectively, one can conclude that the right-hand side of (\ref{alm}) is zero when $\bar{p}_2(z)=-\epsilon p_y\left(z,-\frac{1}{2}\right)$ and $\bar{p}_1(z)=-\epsilon p_y\left(z,\frac{1}{2}\right)$. In the shifted variables $x=z+\frac{1}{2}$ and $\xi=y+\frac{1}{2}$, it follows from (\ref{gainR}), (\ref{kl}) that $P\left(x,\xi\right)=P\left(z+\frac{1}{2},y+\frac{1}{2}\right)=p(z,y)$, and hence, ${p}_2(x)=-\epsilon P_{\xi}\left(x,0\right)$ and ${p}_1(x)=-\epsilon P_{\xi}\left(x,1\right)$, which are the gains defined in (\ref{gain1}) and (\ref{gainp2}), respectively. Finally, differentiating (\ref{r1}) with respect to $z$ and setting $z=\frac{1}{2}$, we get employing (\ref{zn}), (\ref{3rel}) that $\bar{e}_z\left(\frac{1}{2},t\right)=p\left(\frac{1}{2},\frac{1}{2}\right)\bar{w}\left(\frac{1}{2},t\right)=p\left(\frac{1}{2},\frac{1}{2}\right)\bar{e}\left(\frac{1}{2},t\right)$, and hence, relation (\ref{esys2z}) is recovered utilizing the fact that $-p_{11}=P(1,1)=p\left(\frac{1}{2},\frac{1}{2}\right)$, which follows from (\ref{gain2}). One can then derive relations (\ref{esys1z}), (\ref{e sys lot}) from (\ref{z1}), (\ref{zn lot}) via (\ref{r2}), employing identical arguments.

\section*{Acknowledgments}
\textcolor{black}{Nikolaos Bekiaris-Liberis was supported by the funding from the European Commission's Horizon 2020 research and innovation programme under the Marie Sklodowska-Curie grant agreement No. 747898, project PADECOT.}

Rafael Vazquez acknowledges financial support of the Spanish Ministerio de Economia y Competitividad under grant MTM2015-65608-P.

\end{document}